\documentclass[11pt,a4paper]{amsart}

\usepackage{amsmath}
\usepackage{amsfonts}  
\usepackage{amssymb}  
\usepackage{graphicx}  
\usepackage{color}  
\usepackage[T1]{fontenc}
\usepackage{graphpap}
\usepackage{rotating}
\usepackage{hyperref}

\newcommand\aut[1]{[#1]_{\mathbb{T}^2}}

\def\flot{\Phi_{p,q,r}}
\def\flotqr{\Phi_{2,q,r}}
\def\flotpqrs{\Phi_{p,q,r,s}}

\def\Gpqr{\Gamma_{p,q,r}}
\def\Gqr{\Gamma_{2,q,r}}
\def\Gpqrs{\Gamma_{p,q,r,s}}

\def\Hy{\mathbb{H}^2}

\def\Orb{\Hy/\Gpqr}
\def\Orbqr{\Hy/\Gqr}
\def\Orbpqrs{\Hy/\Gpqrs}

\def\Sph{\mathbb S}

\def\SLZ{\mathrm{SL}_2(\Z)}
\def\Sqr{S_{2,q,r}}
\def\Spqr{S_{p,q,r}}
\def\Spqrs{S_{p,q,r,s}}
\def\wSqr{\widehat S_{2,q,r}}
\def\wSpqr{\widehat S_{p,q,r}}
\def\wSpqrs{\widehat S_{p,q,r,s}}

\def\U{T^1}
\def\UOrb{\U\Orb}

\def\Z{\mathbb Z}

\numberwithin{equation}{section}

\newtheorem{theorem}{Theorem}[section]
\newtheorem{lemma}[theorem]{Lemma}

\newtheorem{definition}[theorem]{Definition}

\newtheorem{question}[theorem]{Question}

\newtheorem{thm}{Theorem}
\newtheorem{ques}{Question}
\newtheorem{prop}[thm]{Proposition}
\newtheorem*{coro}{Corollary}

\def\WEN{\color{black}}
\definecolor{darkgreen}{rgb}{0,0.4,0}

\title{Genus one Birkhoff sections for geodesic flows}

\author{Pierre Dehornoy}
\date{20 August 2012}
\address{Mathematisches Institut, Universit\"at Bern, Sidlerstrasse 5, 3012 Bern, Switzerland}
\email{dehornoy@math.unibe.ch}

\begin{document}

\maketitle

\begin{abstract}
We prove that the geodesic flow on the unit tangent bundle to every hyperbolic 2-orbifold that is a sphere with 3 or 4 singular points admits explicit genus one Birkhoff sections, and we determine the associated first return maps. 
\end{abstract}

In this article we investigate from a topological viewpoint a particular family of 3-dimensional flows, namely the geodesic flows on the unit tangent bundle to a hyperbolic 2-orbifold.

Our main subject of investigation is Birkhoff sections. A \emph{Birkhoff section} for a 3-dimensional flow~$\Phi$ is a surface whose boundary is the union of finitely many periodic orbits of~$\Phi$, and such that every other orbit of~$\Phi$ intersects the interior of the surface within a bounded time. Introduced by Poincar\'e and Birkhoff~\cite{Birkhoff}, Birkhoff sections are useful, when they exist, because they provide a description of the flow, minus the orbits forming the boundary of the section, as the suspension flow constructed from the first return map on the section. 

Fried proved~\cite{FriedAnosov} that every transitive 3-dimensional Anosov flow admits Birkhoff sections. (We recall that a 3-dimensional flow associated with a vector field~$X$ on a Riemannian manifold~$M$ is said to be \emph{of Anosov type}~\cite{Anosov} if there exists a decomposition of the tangent bundle~$TM$ as $\langle X\rangle \oplus E^{ss} \oplus E^{su}$, such that $E^{ss}$ is uniformly contracted by the flow and $E^{su}$ is uniformly contracted by the inverse of the flow. A flow is called \emph{transitive} if it admits at least one dense orbit.) Now Fried's proof gives no indication about the complexity of the Birkhoff sections, in particular their genus, and the following is still open:

\begin{ques}[Fried]
\label{Q:Fried}
Does every transitive 3-dimensional Anosov flow admit a genus one Birkhoff section?
\end{ques}

A positive answer, that is, the existence of a genus one Birkhoff section, gives rich information about the considered flow. Indeed, if $S$ is a Birkhoff section for an Anosov flow, the stable and unstable foliations induce two transverse one-dimensional foliations on~$S$ that are preserved by the first return map. Moreover, Fried showed~\cite{FriedAnosov} that the first return map is a pseudo-Anosov diffeomorphism. If $S$ has genus one (and some condition on the boundary of~$S$ is satisfied), the first return map is even an Anosov diffeomorphism, and, therefore, it is conjugated to a linear automorphism of the torus---hence it essentially reduces to an element of~$\SLZ$.  

Here we consider Question~\ref{Q:Fried} for particular 3-dimensional flows, namely the geodesic flows associated with hyperbolic 2-dimensional objects. If $\Gamma$ is a cocompact Fuchsian group~$\Gamma$, the quotient~$\Hy/\Gamma$ of the hyperbolic plane~$\Hy$ is a compact hyperbolic 2-orbifold. Then the geodesic flow on~$\U\Hy=\left\{(x,v)\in T\Hy \mid \vert v\vert = 1\right\}$ projects on~$\U\Hy/\Gamma$. The latter is a 3-dimensional Anosov flow, called the \emph{geodesic flow} on~$\U\Hy/\Gamma$. If no point of~$\Hy$ has a nontrivial stabilizor, $\Hy/\Gamma$ is a hyperbolic surface, and the answer to Question~A for the corresponding geodesic flow is positive: a construction of Birkhoff~\cite{Birkhoff} exhibits two genus one Birkhoff sections. Otherwise, $\Hy/\Gamma$ is a 2-orbifold with singular points, and the only known results about Birkhoff sections for the corresponding flows concern orbifolds which are spheres with three singular points of respective order~$2,3,$ and $4g{+}2$, with $g\ge 2$, for which we proved~\cite{Levogyre} that every collection of periodic orbits bounds a Birkhoff section, but with no control of the genus of the sections. The first aim of the current article is to establish the following general positive result:

\begin{thm}
\label{T:A}
For every hyperbolic 2-orbifold~$\Hy/\Gamma$ that is a sphere with 3 or 4 singular points, the geodesic flow on~$\U\Hy/\Gamma$ admits some (explicit) genus one Birkhoff sections. 
\end{thm}

Note that Theorem~\ref{T:A} essentially covers all spherical orbifolds with at most four singular points. Indeed, there exist 2-orbifolds which are spheres with zero, one or two singular points, but they are either bad orbifolds in the sense of Thurston~\cite{Thurston} or of spherical type. In both cases, the associated geodesic flows do not have the Anosov property we are interested in.

As explained above, when an Anosov flow admits a Birkhoff section of genus one, the associated first return map is particularly simple, namely it is conjugated to an automorphism of the torus, hence it is specified by a conjugacy class in~$\SLZ$. Then the question naturally arises of determining which conjugacy classes occur in this way. For instance, writing~$X$ for~$\left( \begin{smallmatrix} 1&1\\0&1\end{smallmatrix}\right)$ and $Y$ for~$\left( \begin{smallmatrix} 1&0\\1&1\end{smallmatrix}\right)$, and denoting by~$\aut{A}$ the linear automorphism of the torus associated with a $\SLZ$-matrix~$A$, Ghys~\cite{GhysGodbillon}, Hashiguchi~\cite{Hashiguchi} and Brunella~\cite{Brunella} showed that, in the case of a hyperbolic surface of genus~$g$, the first return maps associated with the above mentioned Birkhoff sections are conjugated to the automorphisms~$\aut{(X^2Y^{g-1})^2}$ and~$\aut{X^2(X^2Y^{g-1})^2}$. 

\begin{ques}[Ghys]
\label{Q:Ghys}
Let $A$ be a matrix in~$\SLZ$ with $\mathrm{trace}(A)>2$. Does there exist a hyperbolic 2-orbifold~$\Hy/\Gamma$ such that the geodesic flow on~$\U\Hy/\Gamma$ has a genus one Birkhoff section whose associated first return map is conjugated to~$\aut{A}$?
\end{ques}

We recall (see for instance \cite{LivreAlgebre}) that every element of~$\SLZ$ that is not of finite order is conjugated to a finite product of the matrices~$X$ and~$Y$, and that the product is unique up to cyclically permuting the factors and up to exchanging~$X$ and~$Y$. Moreover, the element is hyperbolic if and only if the product contains both~$X$ and~$Y$. So Question~\ref{Q:Ghys} may be rephrased as a question about which expressions in~$X$ and~$Y$ can occur. We prove:

\begin{thm}
\label{T:B}
For every matrix~$A$ that can be expressed with at least one~$X$ and one~$Y$ and at most four~$X$ or at most four~$Y$, there exists a hyperbolic 2-orbifold~$\Hy/\Gamma$ and a genus one Birkhoff section for the geodesic flow on~$\U\Hy/\Gamma$ whose first return map is conjugated to~$\aut{A}$.
\end{thm}

The matrix $X^{t-2}Y$ has trace~$t$ and, for $t> 2$, it is eligible for Theorem~\ref{T:B}. So we immediately deduce:

\begin{coro}
For every~$t$ larger than~2, there exists a hyperbolic 2-orbifold and a genus one Birkhoff section for its geodesic flow such that the first return map has trace~$t$.
\end{coro}\WEN

Theorems~\ref{T:A} and~\ref{T:B} strengthen the close connection first described by Fried~\cite{FriedAnosov} and Ghys~\cite{GhysGodbillon} between the two main known classes of Anosov flows in dimension~3, namely the geodesic flows on negatively curved orbifolds, and the suspension flows associated with linear automorphisms of the torus. In particular, as they go in the direction of positive answers to Questions~\ref{Q:Fried} and~\ref{Q:Ghys}, a possible interest of these results is to support the idea that there might exist only one transitive Anosov flow up to virtual almost conjugation: two flows~$\phi_1, \phi_2$ on two 3-manifolds~$M_1, M_2$ are said to be \emph{almost conjugated} if there exist two finite collections~$C_1, C_2$ of periodic orbits of~$\phi_1, \phi_2$ such that the manifolds $M_1\setminus C_1$ and~$M_2\setminus C_2$ are homeomorphic and that the induced flows are topologically conjugated, and ``virtually'' is added when the property may involve a finite covering. A positive answer to Question~\ref{Q:Fried} would imply that every transitive Anosov flow is almost conjugated to some suspension flow of a torus automorphism, whereas a positive answer to Question~\ref{Q:Ghys} would imply that any two suspension flows would be virtually almost conjugated. 

At the technical level, Theorems~\ref{T:A} and~\ref{T:B} will come as direct consequences of the following more comprehensive result.

\begin{prop}
\label{T:FirstReturn}
For all $q, r$ with $q\le r $ and $\frac 1 2+\frac 1 q+\frac 1 r<1$, the geodesic flow on the 3-manifold $\U\Orbqr$ admits an explicit Birkhoff section of genus one; the associated first return map is conjugated to 
\[
    \begin{cases}
       \aut{X^{r-6}Y} & \mbox{for } q=3, \\
        \aut{X^{q-4}YX^{r-4}Y} & \mbox{otherwise}.
    \end{cases}
\]

For all $p,q,r$ larger than~$2$ with $\frac 1 p+\frac 1 q+\frac 1 r<1$, the geodesic flow on the 3-manifold $\U\Hy/\Gpqr$ admits two explicit Birkhoff sections of genus one; the associated first return map are conjugated to 
\[
        \aut{X^{p-3}YX^{q-3}YX^{r-3}Y} \mbox{ and } \aut{X^{p-3}YX^{r-3}YX^{q-3}Y}.
\]

For all $p,q,r,s$ larger than~$1$ with $\frac 1 p+\frac 1 q+\frac 1 r+\frac 1 s<2$, the geodesic flow on the 3-manifold $\U\Orbpqrs$ admits six explicit Birkhoff sections of genus one; the associated first return map are conjugated to 
\[ \aut{X^{p-2}YX^{q-2}YX^{r-2}YX^{s-2}Y}
\] 
and to the five other classes obtained by permuting the exponents in the latter expression (up to cyclic permutation the 24 possible permutations give rise to six classes).
\end{prop}

Proposition~\ref{T:FirstReturn} involves three cases. The strategy of proof is the same in every case, but the difficulty is increasing, so that we shall present each proof in a separate section. Starting from the orbifold~$\Orbqr$ (\emph{resp.} $\Orb$, \emph{resp.} $\Orbpqrs$), the idea is to explicitly construct two (\emph{resp.} three, \emph{resp.} four) Birkhoff sections for the geodesic flow, to compute their Euler characteristics (thus checking that their genus is one), and to find a suitable pair of loops on each of them that form a basis of their first homology group. In the first case, these two Birkhoff sections correspond to an avatar of Birkhoff's construction~\cite{Birkhoff}, but, in the other two cases, the method is new. Then, the idea is to start from one Birkhoff section, to follow the geodesic flow until one reaches another section, and to look at how the loops on the first section are mapped on the second one. Our particular choice of the sections will guarantee that this application is described by a simple matrix of the form~$X^iY$. Iterating this observation twice (\emph{resp.} three, \emph{resp.} four times), one obtains the expected form for the first return map.

Let us conclude this introduction with two more remarks about particular cases. First, the case~$p=2, q=3, r=7$ in Proposition~\ref{T:FirstReturn} leads to a toric section with first return map~$\left( \begin{smallmatrix} 2&1\\1&1\end{smallmatrix}\right)$. This matrix is known to correspond to the monodromy of the figure-eight knot. Therefore, after removing one periodic orbit, the geodesic flow of the $(2,3,7)$-orbifold is conjugated to the suspension flow on the complement of the figure-eight knot---one of the first flows whose periodic orbits have been studied from the topological point of view~\cite{BW}. Concerning the topology of the underlying 3-manifolds, this implies that the Seifert fibered space~$\U\Hy/\Gamma_{2,3,7}$ can be obtained from~$\Sph^3$ by a surgery on the figure-eight knot. A celebrated theorem by Thurston~\cite[Theorem 5.8.2]{Thurston} states that, for every hyperbolic knot in~$\Sph^3$, only finitely many surgeries yield a non-hyperbolic 3-manifold. Our construction exhibits such an example for the figure-eight knot.

Finally, when $(p, q, r)$ goes down to the limit values $(2,3,6), (2,4,4)$ or $(3,3,3)$, the corresponding orbifolds with 3 singular points are of Euclidean type. The same situation occurs when $(p,q,r,s)$ reaches the limit value $(2,2,2,2)$. It turns out that the surfaces obtained by extending the construction of Proposition~\ref{T:FirstReturn} still are Birkhoff sections for the corresponding geodesic flows (which are no longer of Anosov type), and that the first return maps are given by the same formulas. But, as can be expected, the associated matrices are of parabolic type, namely they are powers of the matrix~$X$.

\noindent {\bf Acknowledgments:} It is a pleasure to thank \'Etienne Ghys for many discussions on topics related to this article, and Tali Pinski for an invitation and a collaboration which are at the origin of this work.

\section{The case $p=2$: symmetric boundary}
\label{S:p2}

Here we establish the first case in Proposition~\ref{T:FirstReturn}, namely that of an orbifold with three singular points among which one corresponds to an angle~$\pi$. In the whole section, $q, r$ denote two positive integers satisfying $1/2+1/q+1/r<1$. Then $\Orbqr$ is the unique hyperbolic orbifold with three singular points, called~$P, Q, R$, of respective order~$2, q, r$. We denote by~$\U\Orbqr$ its unit tangent bundle, and by~$\flotqr$ the geodesic flow on~$\U\Orbqr$. Abusing notation, we also use $PQR$ to refer to a fixed triangle with respective angles~$\pi/2, \pi/q, \pi/r$  in the hyperbolic plane, and we denote by~$\bar P$ the image of~$P$ under the symmetry around~$QR$. The quadrangle~$PQ\bar PR$ is a fundamental domain for the orbifold~$\Orbqr$. 

We first describe two particular Birkhoff sections for~$\flotqr$, called~$\Sqr^Q$ and~$\Sqr^R$. As we shall see, these two sections turn out to be isotopic. Indeed we will show that, starting from~$\Sqr^Q$ and following~$\flotqr$ for some (non-constant) time, we first reach~$\Sqr^R$, and then come back to~$\Sqr^Q$. 
The second step consists in computing the corresponding homeomorphisms from~$\Sqr^Q$ to~$\Sqr^R$, and from~$\Sqr^Q$ to~$\Sqr^R$. 
The benefit of considering two sections instead of one is that the associated matrices are especially simple (these are companion matrices). The first return map on each of the two sections is then obtained by composing the homeomorphisms.

\subsection{Two Birkhoff sections}

The construction of the surfaces~$\Sqr^Q$ and $\Sqr^R$ that we propose is similar to Birkhoff's original construction~\cite{Birkhoff} of sections for the geodesic flow (although Birkhoff only dealt with surfaces, not with orbifolds), and to A'Campo's construction~\cite{ACampo} of fiber surfaces for divide links. 

\begin{figure}[ht]
   	\begin{minipage}[c]{.2\linewidth}
		\centering
		\begin{picture}(20,23)(0,0)
	 	\put(0,0){\includegraphics*[scale=.45]{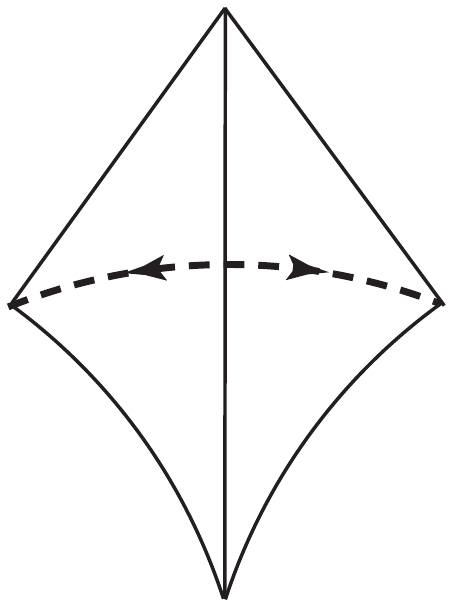}}
		\put(-3,13){$P$}
		\put(21,13){$\bar P$}
		\put(11,27){$Q$}
		\put(11,-0.5){$R$}
		\put(13,16.3){$h$}
		\end{picture}
		\label{F:PossibleShapes}
   	\end{minipage} 
	\hfill
      	\begin{minipage}[c]{.56\linewidth}
		\vspace{-.1cm}
		\hspace{.3cm}
We work in the fundamental domain~$PQ\bar PR$.
Since $P$ corresponds to a singular point of index~$2$, the line~$P\bar P$ is a closed geodesic in the orbifold~$\Orb$: when we reach an end, we just change of direction. Call $h$ this geodesic. It is invariant under the involution that reverses the direction of all tangent vectors. The geodesic ~$h$ divides the orbifolds~$\Orbqr$ into two parts, one containing the singular point~$Q$, and one containing~$R$. We call them the \emph{$Q$-part} and the \emph{$R$-part} respectively.
	\end{minipage}
	\hfill
   	\begin{minipage}[c]{.2\linewidth}
		\centering
		\begin{picture}(20,23)(0,0)
	 	\put(0,0){\includegraphics*[scale=.45]{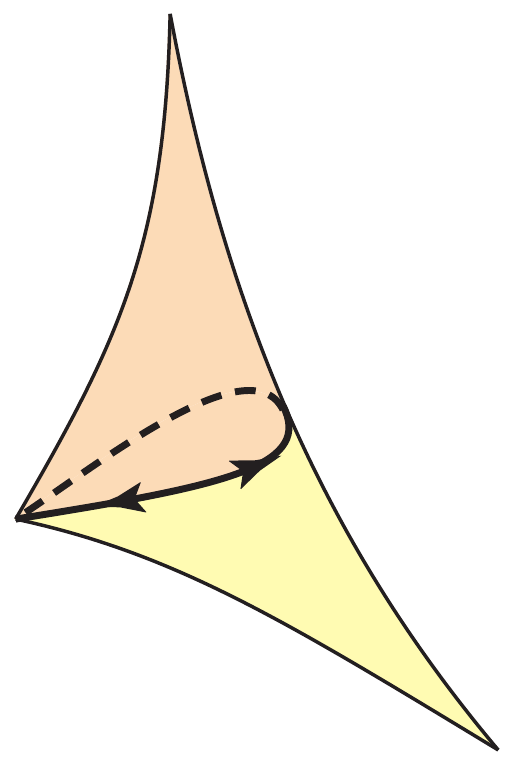}}
		\put(-2,11){$P$}
		\put(9,32){$Q$}
		\put(16,-0.5){$R$}
		\put(8.3,13.2){$h$}
		\end{picture}
		\label{F:PossibleShapes2}
   	\end{minipage} 
\end{figure}

\vspace{-.4cm}

\begin{definition}\label{D:SectionQR}
We call $\Sqr^Q$ (\emph{resp.} $\Sqr^R$) the set, in~$\U\Orbqr$, of all unit tangent vectors to~$h$ that point into the~$Q$-part (\emph{resp.} the $R$-part), plus the whole fiber of the point~$P$. 
\end{definition}

The surfaces~$\Sqr^Q$ and $\Sqr^R$ are topological surfaces. They are both made of one rectangle whose vertical edges are identified using a rotation at~$P$ (see Figure~\ref{F:Order2}). Actually we could smooth them and preserve all their properties at the same time, but we do not need that. The two surfaces intersect along the fiber of~$P$. For~$v$ a unit tangent vector at~$P$, we fix the convention that, at the point~$(P,v)$, the surface~$\Sqr^Q$ is infinitesimally pushed along the direction of~$v$ if~$v$ points into the $Q$-part, whereas~$\Sqr^R$ is infinitesimally pushed along the direction of $v$ if $v$ points into the~$R$-part. This convention is used only for ordering the intersection points of the lift of a geodesic distinct from~$h$ with the two surfaces~$\Sqr^Q$ and~$\Sqr^R$.

\begin{figure}[htb]
	\begin{picture}(100,55)(0,0)
 	\put(0,0){\includegraphics*[scale=.65]{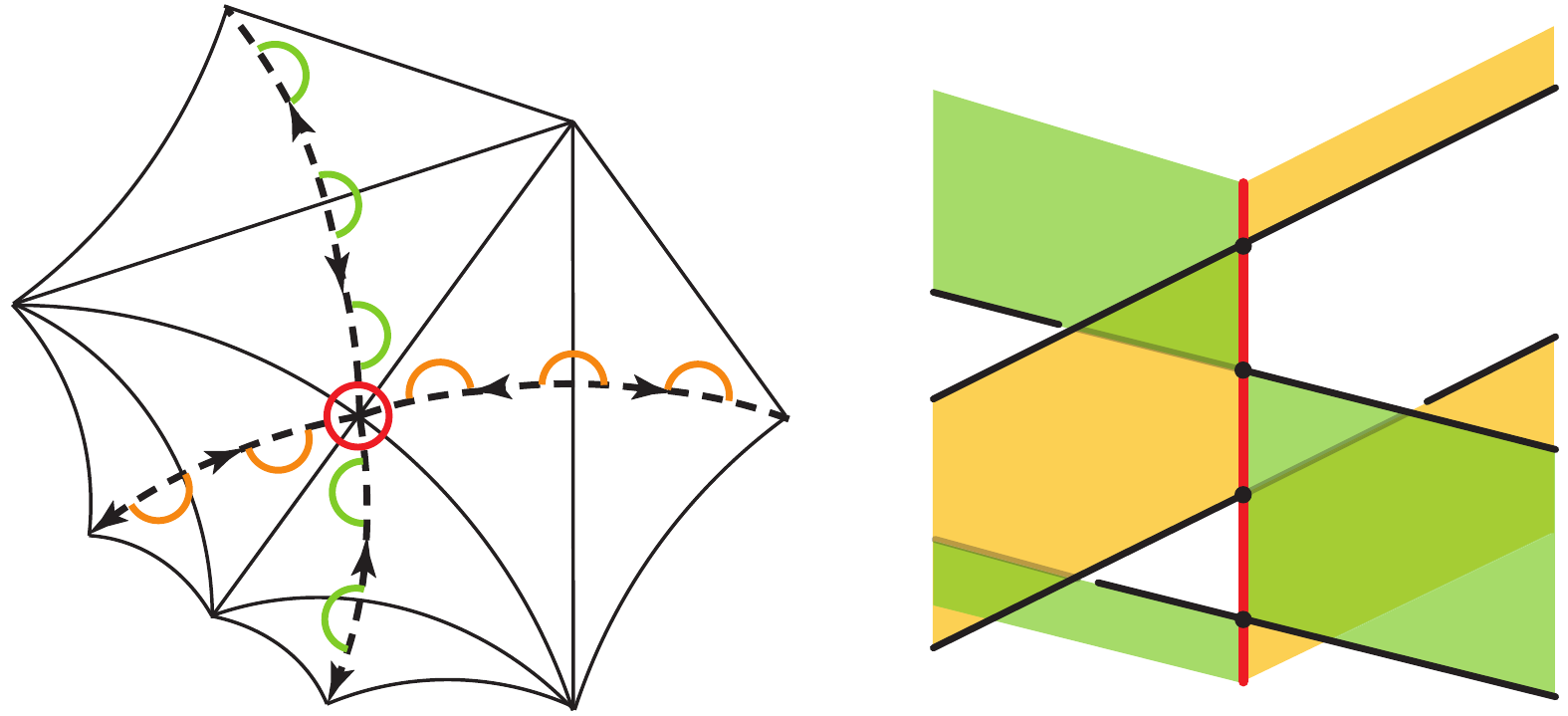}}
	\put(26,17){$P$}
	\end{picture}
	\caption{\small On the left, the union of four fundamental domains for~$\Orbqr$, and four copies of~$\Sqr^Q$. On the right, the surface~$\Sqr^Q$ in a neighbourhood of~$P$, before modding out by the lift of the order~2~rotation around~$P$. It is a topological surface.  Since the lift of the order~2~rotation in the unit tangent bundle is an order~2 screw-motion with no fixed point, we still obtain a topological surface when modding out.}
	\label{F:Order2}
\end{figure}

\begin{lemma}\label{L:S2QR}
The surfaces~$\Sqr^Q$ and~$\Sqr^R$ both have one boundary component, namely the lift~$\tilde h$ of~$h$ in~$\U\Orbqr$. They are genus one Birkhoff sections for the geodesic flow~$\flotqr$. Every orbit of~$\flotqr$ distinct from~$h$ intersects both surfaces alternatively.
\end{lemma}

\begin{proof}
That both surfaces have boundary~$\tilde h$ is clear from the definition. 

The lifts of~$h$ in the universal cover~$\Hy$ of~$\Orbqr$ divide~$\Hy$ into compact $q$-gons and $r$-gons. Let~$g$ be any geodesic not in the~$\Gqr$-orbit of~$h$, and let~$\tilde g$ be the corresponding orbit of~$\flotqr$. Then~$g$ crosses some copy of~$h$ within a bounded time. When, at the intersection point, $g$ goes from an~$r$-gon to a~$q$-gon, $\tilde g$ intersects~$\Sqr^Q$, and when $g$ goes from a~$q$-gon to an~$r$-gon, $\tilde g$ intersects~$\Sqr^R$. Now $g$ could also go through~$P$ directly from a $q$-gon to another one, or from an~$r$-gon to another one. This can only happen above a copy of~$P$, in which case $\tilde g$ intersects both surfaces at the same time. So, in all cases, $\tilde g$ intersects both surfaces within a bounded time. Therefore both~$\Sqr^Q$ and~$\Sqr^R$ are Birkhoff sections. Owing to the convention about the fiber of~$P$, the curve~$\tilde g$ intersects both surfaces alternatively.

As for the genus, since both surfaces consist of one rectangle whose vertical edges are identified using a rotation at~$P$, they are made of one 2-cell, four 1-cells (two horizontal and two vertical) and two 0-cells (in the fiber of~$P$), so their Euler characteristics is~$-1$. Since the surfaces have one boundary component, they are tori.
\end{proof}

\begin{figure}[ht!]
      	\begin{minipage}[c]{.68\linewidth}
		\vspace{-.5cm}
		\hspace{.3cm}
We now define $c^{Q}_+$ (\emph{resp.} $c^R_+$) to be the family of all tangent vectors to~$h$ that points toward~$Q$, and~$c^{Q}_-$ (\emph{resp.} $c^R_-$) to be obtained from~$c^{Q}_+$ (\emph{resp.} $c^R_+$) by reversing the direction of every tangent vector (or, equivalently, by applying a $\pi$-rotation in each fiber). By definition, $c^{Q}_+$ and, similarly, $c^R_+$, $c^{Q}_-$, and~$c^R_-$, consist of elements of~$\U\Orbqr$ that continuously depend on one parameter, hence they are curves in~$\U\Orbqr$, and even loops since $h$ is a closed curve.	\end{minipage}
	\hfill
   	\begin{minipage}[c]{.3\linewidth}
		\centering
		\begin{picture}(30,28)(0,0)
 		\put(0,-2){\includegraphics*[scale=.6]{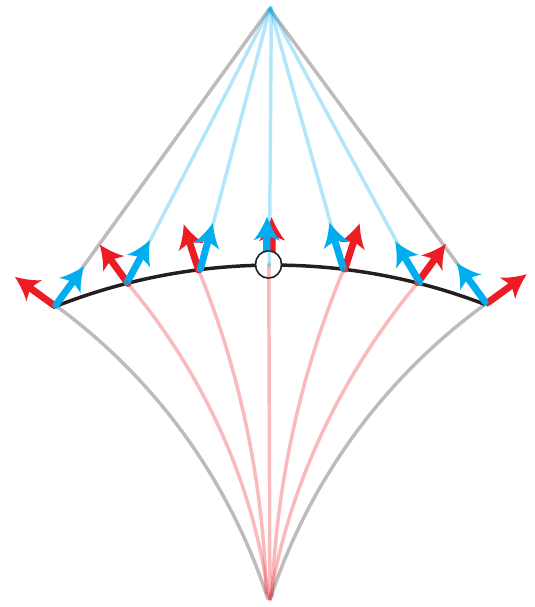}}
		\put(19,32){$Q$}
		\put(18,-1){$R$}
		\put(1,20){\color{blue}{$c^Q_+$}}
		\put(0,12.5){\color{red}{$c^R_-$}}
		\end{picture}
   	\end{minipage} 
\end{figure}

\vspace{-.5cm}

\begin{lemma}\label{L:CurvesQR}
The loops~$c^Q_+$ and $c^R_-$ generate~$\pi_1(\Sqr^Q)$, and $c^R_+$ and $c^Q_-$ generate~$\pi_1(\Sqr^R)$.
\end{lemma}

\begin{proof}
The loops $c^{Q}_+$ and $c^{R}_-$ both lie in~$\Sqr^Q$, and they intersect each other exactly once, namely in the fiber of the intersection between $h$ and $QR$, at the vector that points toward~$Q$. This is enough to ensure that these loops generate the fundamental group of a once-punctured torus. The case of the other pair is similar.
\end{proof}

Hereafter we shall denote by~${\wSqr^Q}$ and ${\wSqr^R}$ the surfaces obtained from $\Sqr^Q$ and~$\Sqr^R$ by compactifying their boundaries to a point. By Lemma~\ref{L:S2QR}, they are compact tori. By Lemma~\ref{L:CurvesQR}, the classes~$[c^Q_+]$ and $[c^R_-]$ generate~$H_1({\wSqr^Q}, \Z)$, whereas $[c^R_+]$ and $[c^Q_-]$ generate~$H_1({\wSqr^R}, \Z)$.

\subsection{First return maps}
For every tangent vector $v$ that lies in the surface~$\Sqr^Q$ and not in the fiber of~$P$, we define $\phi^Q(v)$ to be the first intersection between the orbit of~$\flotqr$ starting from~$v$ and the surface~$\Sqr^R$. For $v$ a tangent vector at~$P$ that points into the~$R$-part, we define~$\phi^Q(v)$ as $v$ itself (seen as an element of~$\Sqr^R$). In this way, we obtain a map~$\phi^Q$ from the whole surface~$\Sqr^Q$ into~$\Sqr^R$. Then~$\phi^Q$ extends to a map from~${\wSqr^Q}$ to~${\wSqr^R}$, denoted by~$\widehat{\phi}^Q$.

\begin{lemma}\label{L:QRreturn}
The map~$\widehat{\phi}^Q$ is a homeomorphism from~${\wSqr^Q}$ to~${\wSqr^R}$ of Anosov type. It is conjugated to the linear homeomorphism whose matrix with respect to the two bases~$(c^Q_+, c^R_-)$ and~$(c^R_+, c^Q_-)$ is~$\left(\begin{smallmatrix} 0 & -1\\{1} &{q-2}\end{smallmatrix}\right)$.
\end{lemma}

\begin{proof}
The continuity of~$\phi^Q$ is clear. Since~$\phi^Q$ is obtained by following the orbits of~$\flotqr$, the injectivity is also clear. Therefore $\phi^Q$ is a homeomorphism, and so is~$\widehat{\phi}^Q$. The weak stable and unstable foliations of the geodesic flow in~$\U\Orbqr$ induce on~$\Sqr^Q$ two foliations which are respectively contracting and expanding. Therefore, by Fried's argument~\cite{FriedAnosov}, $\widehat{\phi}^Q$ is of Anosov type.

For determining the images of~$c^Q_+$ and~$c^R_-$, we unfold~$\Orbqr$ around the point~$Q$ (see Figure~\ref{F:2QR}) by gluing~$q$ copies of the quadrangles~$PQ\bar PR$. We thus obtain a~$q$-gon, denoted by~$R_1\dots R_{q}$ (with $R_1=R$). We denote by~$P_0, \dots, P_{q-1}$ the middle of the edges (with~$P_0=P$ and $P_1=\bar P)$, and for $i=1\dots, g$, we call~$M_i$ the intersection of~$QR_i$ with~$P_{i-1}P_i$.

\begin{figure}[ht]
	\begin{picture}(64,60)(0,0)
	\put(0,0){\includegraphics*[scale=.5]{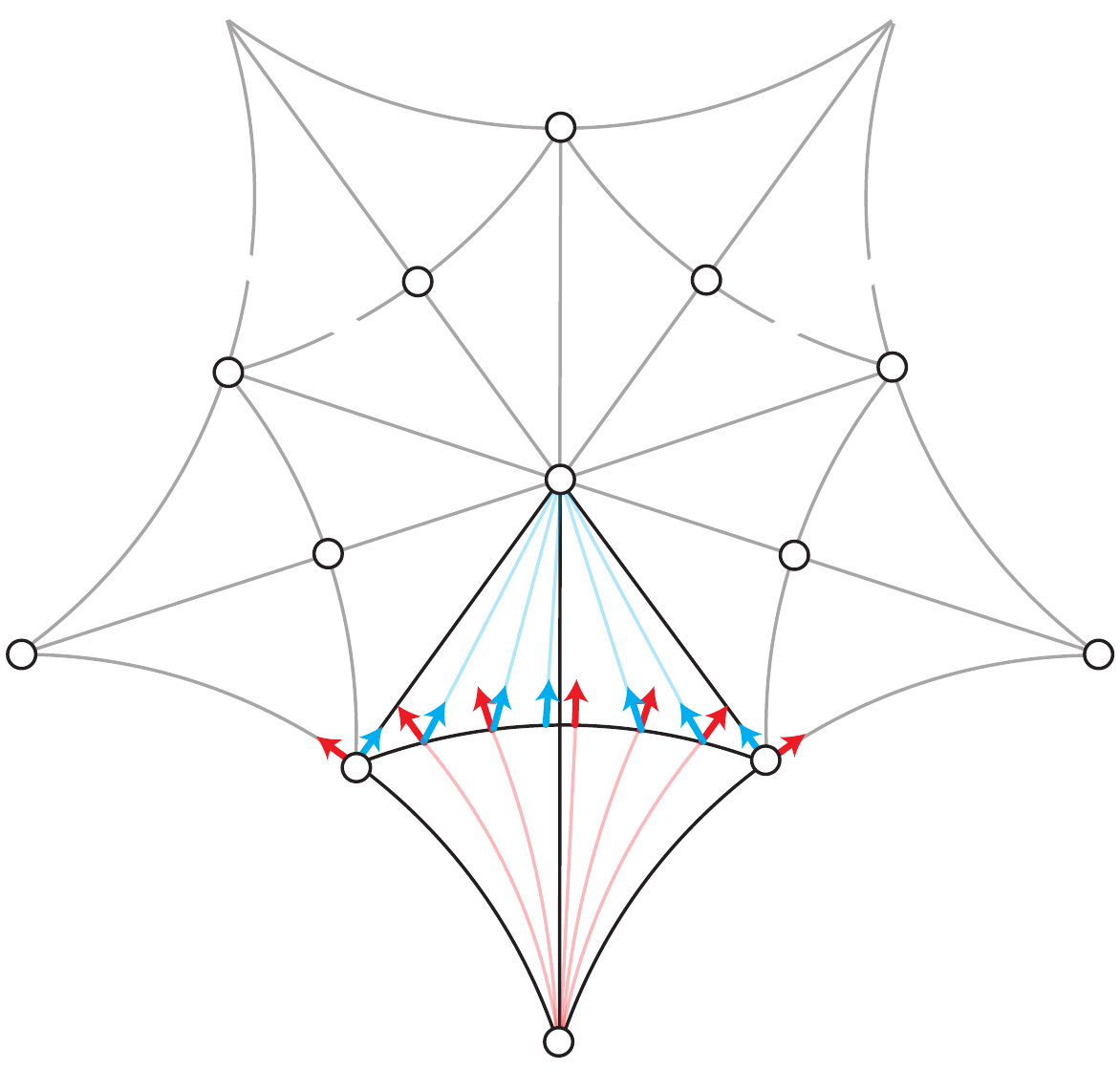}}
	\put(30.5,0){$R_1$}
	\put(59.5,20){$R_2$}
	\put(-3,20){$R_q$}
	\put(43.5,27.5){$M_2$}
	\put(11,28){$M_q$}
	\put(14,14){$P_0$}
	\put(42,14){$P_1$}
	\put(48.5,37){$P_2$}
	\put(3.5,37){$P_{q-1}$}
	\put(30.5,32){$Q$}
	\put(35,44){$M_{({q+1})/2}$}
	\put(15,44.5){$M_{({q+3})/2}$}
	\put(24,52.5){$P_{({q+1})/2}$}
	\end{picture}

	\vspace{-.3cm}
	\begin{picture}(128,55)(0,0)
 	\includegraphics*[scale=.5]{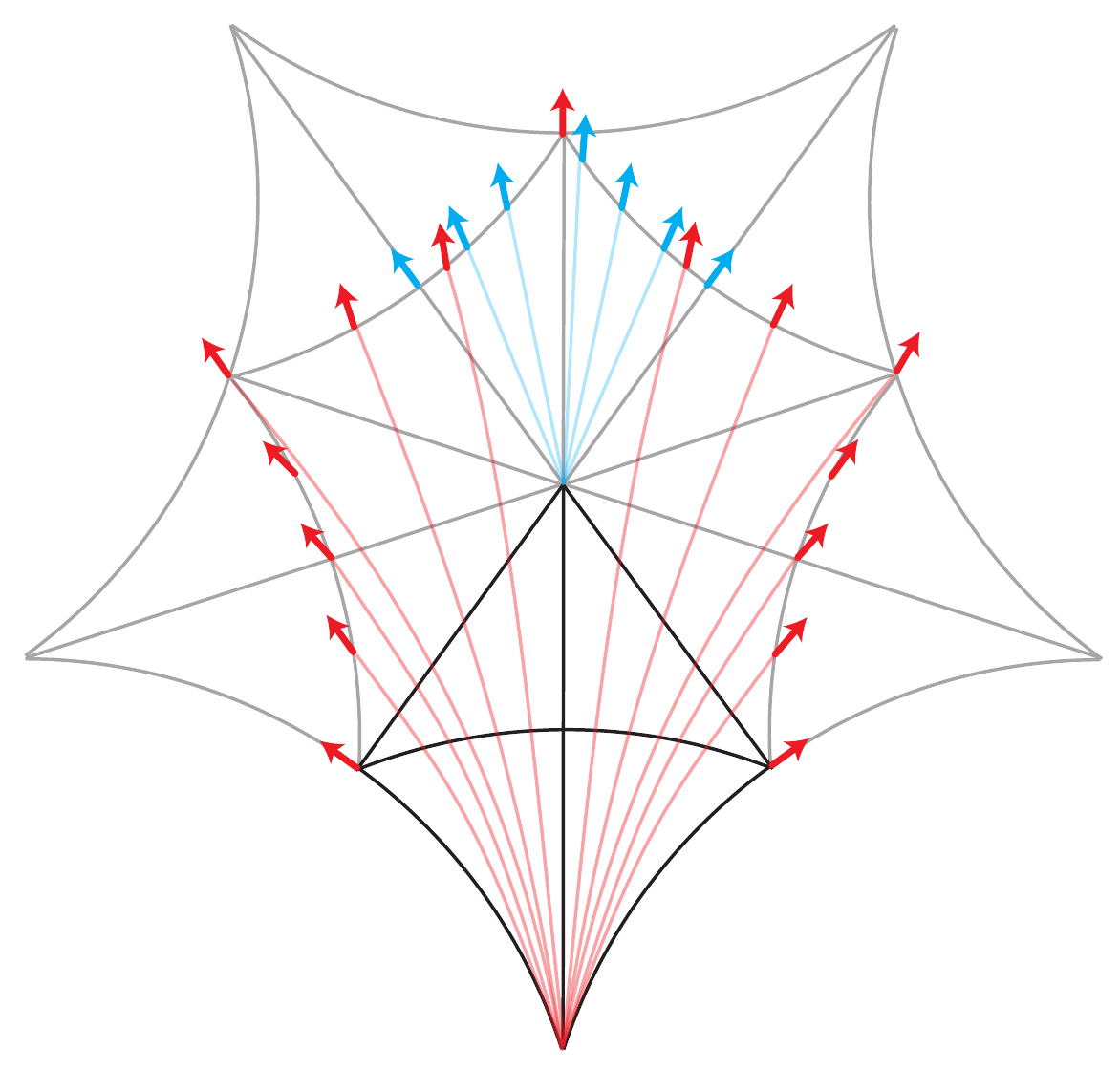}
 	\includegraphics*[scale=.5]{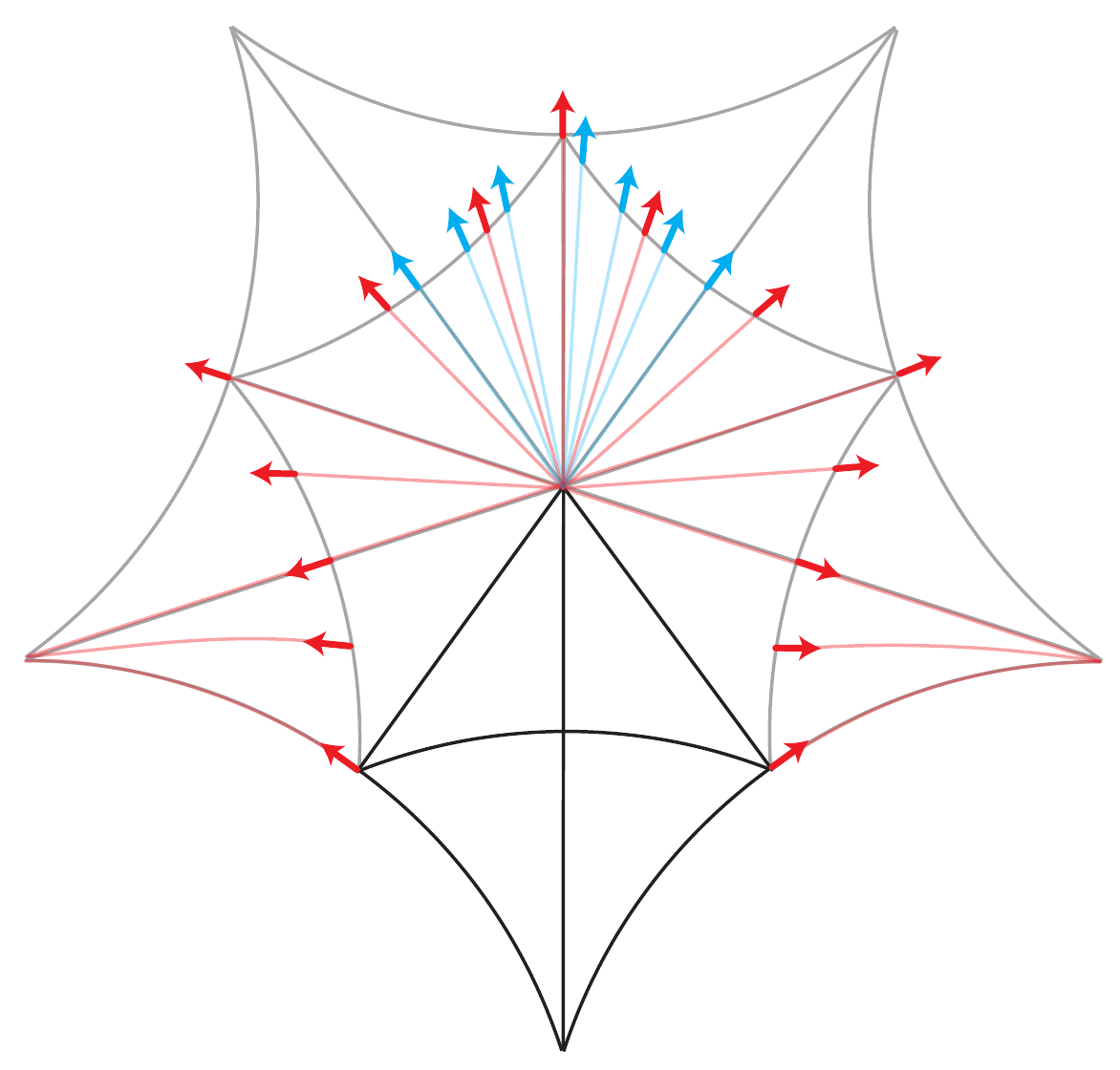}
	\end{picture}
	\caption{\small } 
	\label{F:2QR}
\end{figure}

The curve~$c^Q_+$ is made of those vectors above~$P_0P_1$ that point toward~$Q$. Therefore, under the geodesic flow, $c^Q_+$ enters the $q$-gon~$P_0\dots P_{q-1}$. It reaches first the fiber of~$Q$, and then continues until it reaches (the fibers of the points of) the side(s) opposite to~$P_0P_1$ in~$P_0\dots P_{q-1}$. At that moment, it points into the~$R$-part, and therefore belongs to~$\Sqr^R$: this is~$\phi^Q(c^Q_+)$. Depending on the parity of~$q$, this curve lies in the fibers of the segment~$P_{{q}/2}P_{({q+2})/2}$ or in the fibers of~$M_{({q-1})/2}P_{({q+1})/2}$ and~$P_{({q+1})/2}M_{({q+3})/2}$ (see Figure~\ref{F:2QR}). In both cases, it projects, after modding out by the rotation at~$Q$, to the curve~$c^Q_-$, yielding $\phi^Q(c^Q_+)=c^Q_-$.

For~$\phi^Q(c^R_-)$, we see that the vector of~$c^R_-$ lying above~$P$ (which points to~$R$) is fixed by~$\phi^Q$, by definition. When~$\phi^Q$ is applied, that is, when we follow the flow~$\flotqr$, the rest of~$c^R_-$ goes through the $q$-gon~$P_0\dots P_{q-1}$. Therefore $\phi^Q(c^R_-)$ is made of all the vectors in the fibers of~$P_0\dots P_{q-1}$ that do not lie above~$P_0P_1$ and whose opposite points toward~$R_1$ (see~Figure~\ref{F:2QR} bottom left). Since we are only interested in the class of $\phi^Q(c^R_-)$ in the first homology group, we can apply any convenient isotopy to~$\phi^Q(c^R_-)$ inside~$\Sqr^R$. We thus rotate all vectors (without changing the fibers in which they lie) in the following way: the vectors in the fibers of the segment~$P_0M_g$ are rotated so that they point toward $R_g$ (note that the vector at~$P_0$ does not change), similarly the vectors in the fibers of~$P_1M_2$ are rotated so that they point toward~$R_2$, and all other vectors (corresponding to vectors in the fibers of~$M_2P_2$, or $P_2P_3$, or \dots, or $P_{g-2}P_{g-1}$, or $P_{g-1}M_g$) are rotated so that their opposite points toward~$Q$. The obtained curve (see~Figure~\ref{F:2QR} bottom right) is equal, in the orbifold~$\Orbqr$, to the concatenation of the opposite of $c^R_+$ and $q{-}2$ times~$c^Q_-$. Indeed the parts above~$P_1M_2$ and $P_1M_2$, with the given orientation, add up to the opposite of~$c^R_+$, and the parts above~$M_2P_2P_3\dots P_{g-2}P_{g_1}M_g$ add up to $q{-}2$ times~$c^Q_-$. 
\end{proof}

Arguing similarly, we define a map~$\phi^R$ from~$\Sqr^R$ to~$\Sqr^Q$ in a way that is exactly symmetric to what we did for~$\phi^Q$. 

\begin{proof}[Proof of Proposition~\ref{T:FirstReturn} (first case)]
Consider the two Birkhoff sections~$\Sqr^Q$ and $\Sqr^R$ given by Definition~\ref{D:SectionQR} and Lemma~\ref{L:S2QR}. 
By Lemma~\ref{L:S2QR}, starting from any point of~$\Sqr^Q$ and following~$\flotqr$ for some time (which is bounded, but not the same for all points), we reach the surface~$\Sqr^R$, and then reach~$\Sqr^Q$ again. Therefore the first return map on~$\Sqr^Q$ is obtained by applying~$\phi^Q$ first and then~$\phi^R$. When compactifying, we obtain a Anosov diffeomorphism which is the product of~$\widehat{\phi}^Q$ and~$\widehat{\phi}^R$. Since an Anosov diffeomorphism of the torus is always conjugated to its action on homology, by Lemma~\ref{L:QRreturn}, the first return map is conjugated to the product~$\left(\begin{smallmatrix} 0 & -1 \\ 1 & r-2 \end{smallmatrix} \right) \left(\begin{smallmatrix} 0 & -1 \\ 1 & q-2 \end{smallmatrix} \right)$. 

Since every matrix of the form~$\left(\begin{smallmatrix} 0 & -1 \\ 1 & t \end{smallmatrix} \right)$ is equal to~$X^{-1}YX^{t-1}$, the previous product is equal to $X^{-1}YX^{r-3}X^{-1}YX^{q-3}$, which is conjugated to $X^{q-4}YX^{r-4}Y$. For $q=3$, the exponent~$q-4$ is negative, and therefore the formula can be simplified. An easy computation leads then to $X^{r-6}Y$.
\end{proof}

\section{The case $p\ge 3$: non-symmetric boundary} 
\label{S:pqr}

We now turn to the second case in Proposition~\ref{T:FirstReturn}, namely that of an orbifold with three singular points of index larger than~$2$. In the whole section, $p, q, r$ denote three integers larger than 2 and satisfying $1/p+1/q+1/r<1$. Then $\Orb$ is the unique hyperbolic orbifold with three singular points, called~$P, Q, R$, of respective order~$p, q, r$. We denote by~$\U\Orb$ its unit tangent bundle, and by~$\flot$ the geodesic flow on~$\U\Orb$. As in Section~\ref{S:p2}, we fix a triangle~$PQR$ in the hyperbolic plane with respective angles~$\pi/p, \pi/q, \pi/r$, and we denote by~$\bar P$ the image of~$P$ under the symmetry around~$QR$. The quadrangle~$PQ\bar PR$ is a fundamental domain for the orbifold~$\Orb$. 

The idea is to mimic the two steps of the case~$p=2$, using three surfaces instead of two. The surfaces we use here cannot be described using Birkhoff's construction: when we start from a collection of closed geodesics and try to apply the latter, the obtained surface has genus at least two. On the other hand, the second step is similar to that of Section~\ref{S:p2}, with three maps instead of two.

\begin{figure}[ht!]
      	\begin{minipage}[c]{.68\linewidth}
		\vspace{-.4cm}
		\hspace{.3cm}
\subsection{Three Birkhoff sections}

Since the triangle~$PQR$ is acute (that is, all angles are smaller than~$\pi/2$), there exists in~$PQR$ a closed billiard trajectory of period~3, which bounces every edge once. Call $A, B, C$ the bouncing points with the segments~$QR, RP, PQ$ respectively, and call~$\bar B$ and $\bar C$ the images of~$B$ and $C$ under the symmetry around $QR$. Then there is a closed geodesic, that we denote by~$b$ (for billiard),  travelling from~$B$ to $C$, then from $\bar C$ to $B$, from $\bar B$ to $\bar C$, and finally from $C$ to~$\bar B$. There is another one travelling in the other direction, but we do not consider it now.
		\vspace{-.4cm}
	\end{minipage}
	\hfill
   	\begin{minipage}[c]{.3\linewidth}
		\centering
		\begin{picture}(25,32)(0,0)
 		\put(-5,2){\includegraphics*[scale=.6]{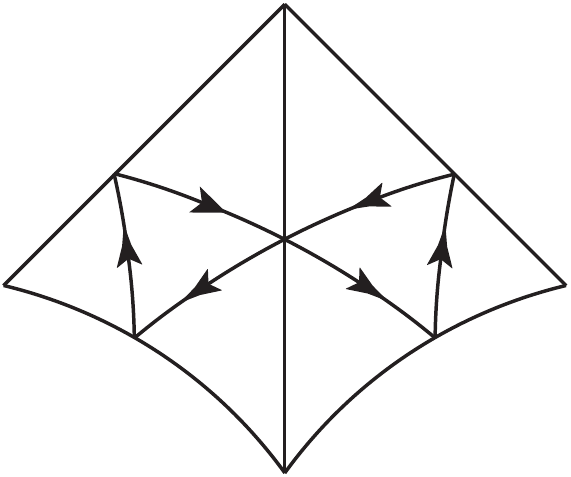}}
		\put(-6,15){$P$}
		\put(29,15){$\bar P$}
		\put(13,31){$Q$}
		\put(13,1){$R$}
		\put(12.5,18){$A$}
		\put(1.5,7){$B$}
		\put(0,21.5){$C$}
		\put(21,7){$\bar B$}
		\put(23,21.5){$\bar C$}
		\end{picture}
   	\end{minipage} 
\end{figure}

We now adapt Birkhoff's construction and describe a surface whose boundary in~$\UOrb$ is the lift~$\tilde b$ of~$b$. First we observe that $b$ divides the orbifold~$\Orb$ into five regions: three regions that contain one of the singular points, hereafter called the $P$-, the $Q$-, and the $R$-parts, and two triangles delimited by~$b$, hereafter called the \emph{orthic triangles}. Abusing notation, we use the same names when working in the universal cover of~$\Orb$.

\begin{figure}[ht!]
   	\begin{minipage}[c]{.2\linewidth}
		\centering
		\begin{picture}(20,19)(0,0)
 		\put(-2,0){\includegraphics*[scale=.45]{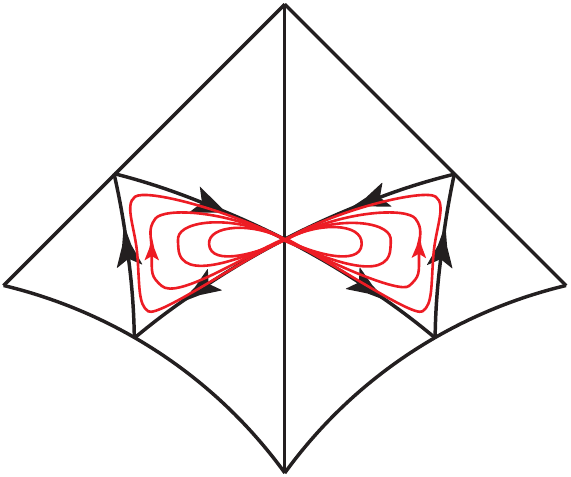}}
		\end{picture}
   	\end{minipage} 
	\hfill
      	\begin{minipage}[c]{.76\linewidth}
		\vspace{-.3cm}
		\hspace{.3cm}
In the fundamental domain~$PQ\bar PR$, the geodesic ~$b$ has the shape of a butterfly centered at~$A$. We now consider a one-parameter family of curves~$(\alpha_s(t))_{s,t\in(0,1)}$ whose union foliates the two orthic triangles, so that every curve~$\alpha_s$ is a smooth butterfly centered at~$A$, and that the butterflies are convex inside each of the orthic triangles.
	\end{minipage}
\end{figure}

\vspace{-.4cm}

\begin{definition}\label{D:Section}
The surface~$\Spqr^A$ is the closure of the set of all unit vectors positively tangent to the family~$(\alpha_s)_{s\in(0,1)}$, that is, the set $\big\{\frac{d\alpha_s(t)/dt}{\Vert d\alpha_s(t)/dt\Vert} \mid s, t\in(0,1)\big\}$.
\end{definition}

Taking the closure of~$\Spqr^A$ is equivalent to adding to~$\Spqr^A$ the tangent vectors to~$b$, the vectors at~$B$ that point into the~$P$-part, the vectors at~$C$ that point into the~$Q$-part, and the vectors at~$A$ that point into the~$Q$-part or into the orthic triangles. 

Next, choosing two other foliations~$\beta_s$ and $\gamma_s$ of the orthic triangles by convex butterflies centered at~$B$ and $C$ respectively, we similarly define surfaces~$\Spqr^B$ and~$\Spqr^C$ as the closures of the set of all unit vectors positively tangent to~$\beta_s$ and $\gamma_s$. 

Now, we introduce~$c^Q_+$ to be the set of all vectors of~$\Spqr^A$ that point directly toward the point~$Q$ and $c^R_-$ to be the set of all vectors~$\Spqr^A$ whose opposite point directly toward~$R$. Then, as in Section~\ref{S:p2}, $c^Q_+$ and~$c^R_-$ are loops in~$\U\Orb$. The loops $c^P_+$ and $c^R_+$ in~$\Spqr^B$ and $c^P_-$ and $c^Q_-$ in $\Spqr^C$ are defined similarly (see Figure~\ref{F:PQR}).

\begin{figure}[ht]
	\begin{picture}(200,80)(0,0)
	\put(0,0){\includegraphics*[width=\textwidth]{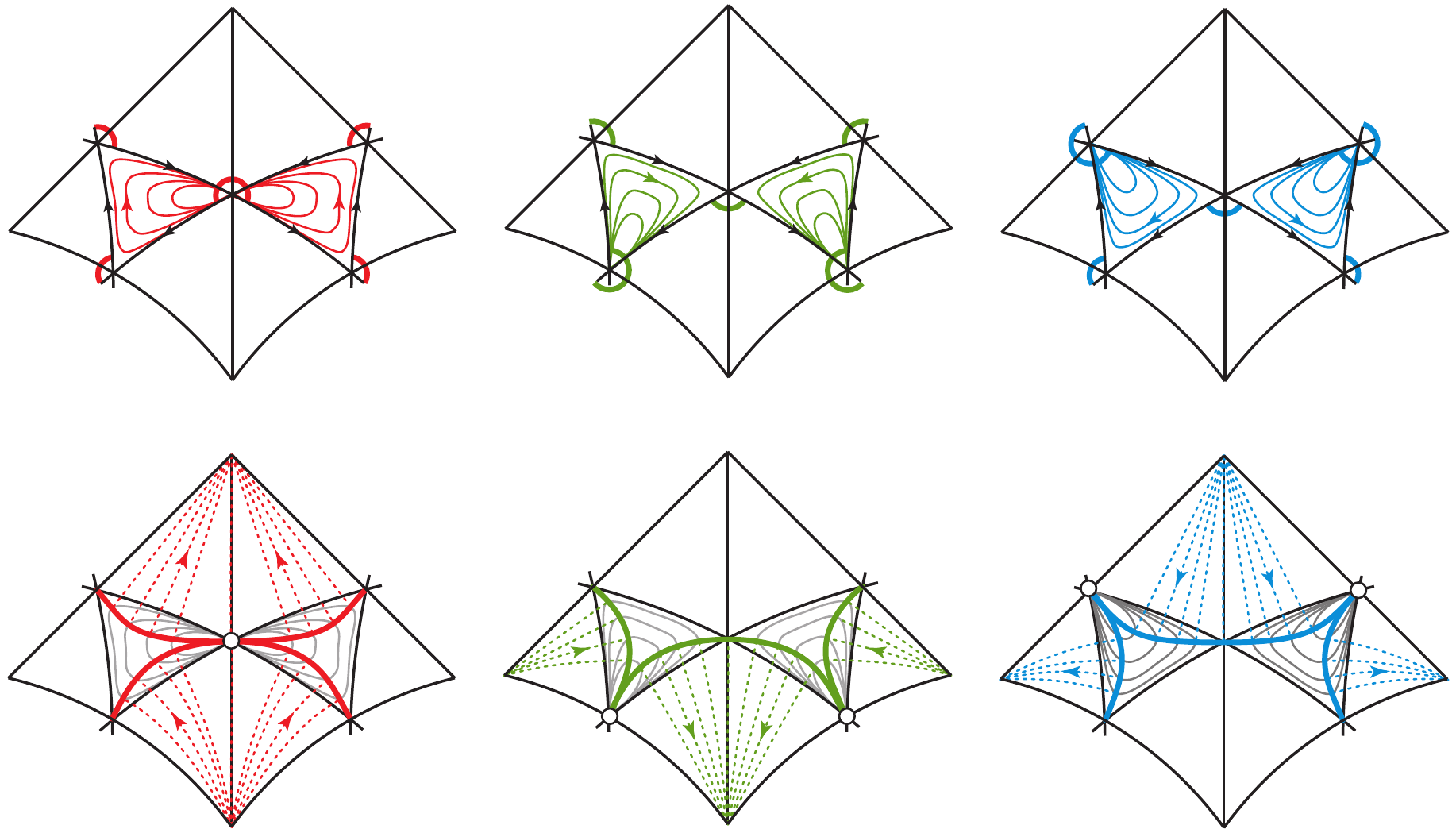}}
	\put(3,67){$\Spqr^A$}
	\put(47,67){$\Spqr^B$}
	\put(90,67){$\Spqr^C$}
	\put(7,25){$c^Q_+$}
	\put(9,4){$c^R_-$}
	\put(53,3.5){$c^R_+$}
	\put(46,8){$c^P_-$}
	\put(94,25){$c^Q_-$}
	\put(89,8){$c^P_+$}
	\end{picture}
	\caption{\small } 
	\label{F:PQR}
\end{figure}

\begin{lemma}\label{L:SPQR}
The surfaces~$\Spqr^A$, $\Spqr^B$ and~$\Spqr^C$ have one boundary component, namely the lift~$\tilde b$ of~$b$. All three surfaces are genus one Birkhoff sections for~$\flot$. The two curves~$(c^Q_+, c^R_-)$ (\emph{resp.} $(c^R_+, c^P_-)$, \emph{resp.} $(c^P_+, c^Q_-)$) form a basis of~$H_1(\Spqr^A, \Z)$ (\emph{resp.} $H_1(\Spqr^B, \Z)$, \emph{resp.} $H_1(\Spqr^C, \Z)$).
\end{lemma}

\begin{proof}
Let~$\gamma$ be a geodesic in~$\Hy$ that is not in the~$\Gpqr$-orbit of~$b$. Define the \emph{code}~$c(\gamma)$ of $\gamma$ to be the bi-infinite word in the alphaber~$\{P, Q, R\}$ describing the different types of the regions crossed by $\gamma$: we write $P$, $Q$ or $R$ when~$\gamma$ goes through the interior of a $P$-part, a $Q$-part, or a $R$-part, and do not write anything when $\gamma$ goes through any of the two types of orthic triangle (this coding is neither injective nor surjective, but this is of no importance). Since $\gamma$ is not in the orbit of~$b$, it cannot cross more than two consecutive orthic triangles. Therefore, the code is indeed bi-infinite. Now, for every factor~$RP, RQ, PQ, PP, QQ$ or $RR$, there exists exactly one point where~$\gamma$ is tangent to the family~$\alpha_s$ and contributes that factor to~$c(\gamma)$. Since every bi-infinite word contains infinitely many such factors, the lift of~$\gamma$ in~$\U\Orb$ intersects $\Spqr^A$ infinitely many times.

For the genus, we observe that $\Spqr^A$ is made of the closures of two discs, corresponding to the lifts of the two orthic triangles. With the considered decomposition, there are six edges corresponding to the different segments of~$b$, plus three edges in the fibers of~$A,B$ and $C$. There are also six vertices, two in each of the fibers of~$A,B$ and $C$. Adding the contributions, we obtain $-1$ for the Euler characteristics of~$\Spqr^A$, and therefore one for its genus.
For the basis of~$H_1(\Spqr^A)$, as in Lemma~\ref{L:S2QR}, we observe that the considered loops intersect each other once.

The proof for the other two surfaces is similar.
\end{proof}

\subsection{First return maps}

We now mimic the construction in Section~\ref{S:p2} of the maps~$\phi^Q$ and~$\phi^R$, with one difference: the maps~$\phi^A$, $\phi^B$ and $\phi^C$ to be defined all have fixed points (but none in common), which correspond to the intersection points between the surfaces $\Spqr^A$, $\Spqr^B$ and $\Spqr^C$. For every tangent vector $v$ that lies in the surface~$\Spqr^A$ and not in the fiber of~$B$, we define $\phi^A(v)$ to be the first intersection between the orbit of the geodesic flow starting from~$v$ and the surface~$\Spqr^C$. For $v$ a tangent vector at~$B$ that points into the~$P$-part, we define~$\phi^A(v)$ as $v$ itself. 

\begin{lemma}\label{L:PQRreturn}
The map~$\phi^A$ is a homeomorphism from the torus~$\Spqr^A$ to the torus~$\Spqr^C$. It is conjugated to the linear homeomorphism whose matrix with respect to the bases~$(c^Q_+, c^R_-)$ and~$(c^P_+, c^Q_-)$ is~$\left(\begin{smallmatrix} 0 & -1\\{1} &{q-1}\end{smallmatrix}\right)$.
\end{lemma}

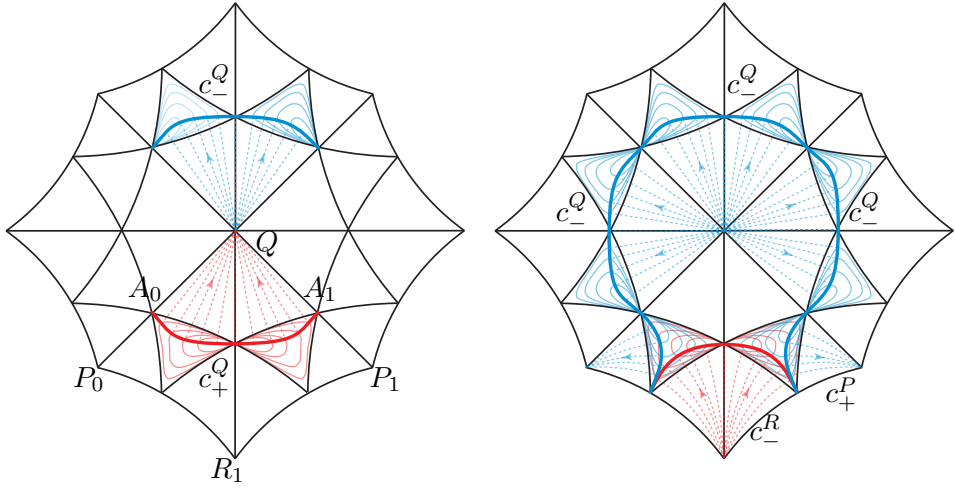
\begin{figure}[ht]
	\begin{picture}(140,70)(0,0)
	\put(0,0){\includegraphics*[width=\textwidth]{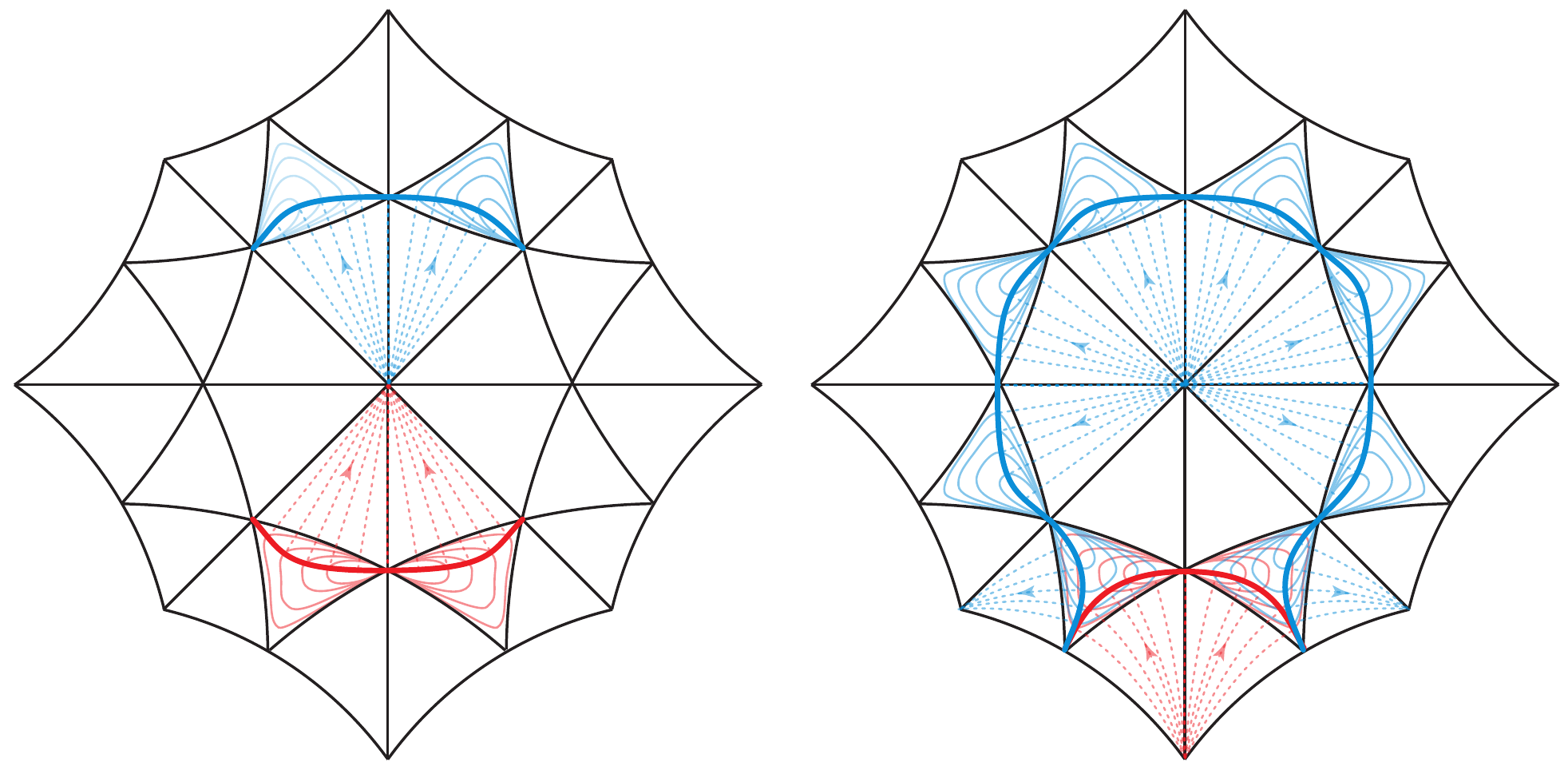}}
	\put(34,28){$Q$}
	\put(10,10){$P_0$}
	\put(28,-2){$R_1$}
	\put(49,10){$P_1$}
	\put(17.2,22){$A_0$}
	\put(40.2,22){$A_1$}
	\put(26.5,10){$c^Q_+$}
	\put(26.5,49){$c^Q_-$}
	\put(99,3){$c^R_-$}
	\put(73.5,32){$c^Q_-$}
	\put(96,49){$c^Q_-$}
	\put(112,32){$c^Q_-$}
	\put(109,8){$c^P_+$}
	\end{picture}
	\caption{\small The images of~$c^Q_+$ and $c^R_-$ under~$\phi^A$ in the case $p=3$, $q=4$, $r=5$. On the left, we see that~$c^Q_+$ is mapped to~$c^Q_-$. On the right, $c^R_-$ is mapped to a long curve, which is homologous to $q-1$ times~$c^Q_-$ plus the opposite of~$c^P_+$.} 
	\label{F:phiPQR}
\end{figure}

\begin{proof}
The argument is similar to the one for Lemma~\ref{L:QRreturn}. The continuity and the injectivity of~$\phi^A$ need no new argument.

We now unfold~$\Orb$ around the point~$Q$ by gluing~$q$ copies of the quadrangle~$PQ\bar PR$. We obtain a~$2q$-gon, that we denote by~$P_0R_1P_1\dots P_{q-1}R_{q}$ (with $R_1=R$, $P_0=P$ and $P_1=\bar P)$, or simply by $P\!ol^Q$. We similarly denote by~$A_i$ and~$C_i$ the corresponding images of~$A$ and $C$. 

We now determine the first intersections of the orbits of~$\flot$ that start on the curves~$c^Q_+$ and $c^R_-$ in~$P\!ol^Q$ with the surface~$\Spqr^C$. By definition, these curves are the images of~$c^Q_+$ and $c^R_-$ under~$\phi^A$.

The curve~$c^Q_+$ is made of those unit vectors tangent to the family~$\alpha_s$ that point toward~$Q$. Therefore, when following the flow~$\flot$, the points of $c^Q_+$ enter the $Q$-part. They first reach the fiber of~$Q$, and then continue on the other side of~$Q$ until they reach the orthic triangles opposed to the starting ones in the $Q$-part. At that moment, the orbits intersect the surface~$\Spqr^C$ when the geodesics they are following are tangent to the family~$\gamma_s$. They form then a curve that connects~$C_{q/2}$ to $A_{(q+2)/2}$ and then to $C_{(q+2)/2}$ if $q$ is even and $A_{(q+1)/2}$ to $C_{(q+1)/2}$ to $A_{(q+3)/2}$ if $q$ is odd (see Figure~\ref{F:phiPQR} left). In both cases, the curve we obtain projects, in~$\U\Orbqr$, to the loop~$c^Q_-$, yielding $\phi^A(c^Q_+)=c^Q_-$.

As for~$\phi^A(c^R_-)$, we see that the (unique) vector of~$c^R_-$ lying above~$B$ (which, by definition, points toward~$R$) is fixed by~$\phi^A$. When $\phi^A$ is applied, that is, when we follow the geodesic flow, the two parts of~$c^R_-$ that are close to~$B$ (one delimited by~$B$ and the intersection of~$RC$ with~$c^R_-$ in the orthic triangle~$ABC$ and one delimited by~$\bar B$ and the intersection of~$R\bar C$ with~$c^R_-$ in~$A\bar B\bar C$) stay in their respective orthic triangles and become tangent to the family~$\gamma_s$ on two curves that join~$B$ to~$C$ and $\bar B$ to~$\bar C$ respectively. When the orientation is taken into account, the union of the latter two curves is isotopic to~$-c^P_+$ (see Figure~\ref{F:phiPQR} right).

The rest of~$c^R_-$ goes through the~$Q$-part and become tangent to~$\gamma_s$ above $2q-2$ curves that respectively connect $C_1$ to~$A_2$, $A_2$ to~$C_2$, \dots, $A_q$ to~$C_0$ in the corresponding orthic triangles. Each of these curves can be deformed by isotopy to the corresponding image of~$c^Q_-$, so that, in the quotient orbifold, their union is isotopic to $q-1$ times~$c^Q_-$. 
\end{proof}

Using an exactly similar construction, we define a map~$\phi^C$ from~$\Spqr^C$ to~$\Spqr^B$ and a map~$\phi^B$ from~$\Spqr^B$ to~$\Spqr^A$. We can now complete the argument.

\begin{proof}[Proof of Proposition~\ref{T:FirstReturn} (second case)]
Consider the three Birkhoff sections $\Spqr^A$, $\Spqr^B$ and $\Spqr^C$ given by Lemma~\ref{L:SPQR}. We argue as in Section~\ref{S:p2}. Starting from any point of~$\Spqr^A$ and following the geodesic flow for some bounded time, we reach the surface~$\Spqr^C$. When continuing, we then reach~$\Spqr^B$, and then $\Spqr^A$ again. Therefore the first return map on~$\Spqr^A$ is obtained by applying~$\phi^A$ first, then~$\phi^C$, and then~$\phi^B$. In terms of matrices, and in the basis~$(c^Q_+, c^R_-)$ of~$\Spqr^A$, that map is then conjugated to the product $\left(\begin{smallmatrix} 0 & -1 \\ 1 & r-1 \end{smallmatrix} \right) \left(\begin{smallmatrix} 0 & -1 \\ 1 & p-1 \end{smallmatrix} \right) \left(\begin{smallmatrix} 0 & -1 \\ 1 & q-1 \end{smallmatrix} \right)$. In terms of the standard generators~$X$ and~$Y$ of~$\SLZ$, the latter product is equal to $X^{-1}YX^{r-2}X^{-1}YX^{p-2}X^{-1}YX^{q-2}$, which is conjugated to $X^{p-3}YX^{q-3}YX^{r-3}Y$, as announced.

The other genus one Birkhoff sections are obtained by reversing the direction of the geodesic~$b$. 
By the same construction as above, we obtain three other Birkhoff sections, and the same arguments lead to a first resturn map conjugated to~$X^{p-3}YX^{r-3}YX^{q-3}Y$. This result can also be obtained  directly by observing that the new Birkhoff sections are obtained from the old ones by rotating all tangent vectors by an angle~$\pi$, and therefore by following the flow~$\flot$ in the reverse direction, so that the new monodromy is the inverse of the old one.
\end{proof}

\section{The case of four singular points}

We now turn to the last case in Proposition~\ref{T:FirstReturn}. Let $p, q, r, s$ denote four integers larger than 2. The case when some of them is equal to~$2$ requires some slight modifications that we will describe at the end of the section. The proof follows the same scheme as in the case of three singular points, with some modifications that make it more complicated. The most notable one is that the two orthic triangles are replaced by two quadrangles and that the boundaries of the constructed Birkhoff tori have two components.

In distinction to the case of three singular points, there exist many orbifolds with spherical base and four singular points of respective orders $p, q, r, s$. Indeed, Thurston~\cite{Thurston} showed that the associated Teichm\"uller space has dimension~2. Nevertheless, the argument of Ghys for surfaces~\cite{GhysConjugues} still applies, so that the associated geodesic flows all are conjugated. Therefore, it is enough to consider here one orbifold for every choice of~$p, q, r, s$. 

From now on, we fix a Fuchsian group~$\Gpqrs$ such that the quotient orbifold~$\Orbpqrs$ has four singular points~$P, Q, R, S$ of respective orders~$p, q, r, s$. We call~$\flotpqrs$ the geodesic flow on~$\U\Orbpqrs$.
We also choose a fundamental domain for~$\Orbpqrs$ in~$\Hy$, obtained by cutting the orbifold along the shortest geodesic connecting \WEN $Q$ to $P$, $P$ to $S$, and $S$ to $R$. The fundamental domain is therefore an hexagon, that we denote by~$P_1QP_2S_2RS_1$. The total angles are $2\pi/q$ at~$Q$ and $2\pi/r$ at~$R$, while the sum of the angles at~$P_1$ and $P_2$ is~$2\pi/p$ and the sum at~$S_1$ and $S_2$ is~$2\pi/s$ (see Figure~\ref{F:sectionPQRS} bottom).

\subsection{Four Birkhoff sections}

Let~$b_1$ be the oriented closed geodesic whose projection connects the segment~$P_1S_1$ to $P_2S_2$, and $b_2$ be the oriented closed geodesic whose projections connects~$P_2S_2$ to~$QP_2$, then $QP_1$ to~$RS_1$, then $RS_2$ to~$QP_1$, then $QP_2$ to~$RS_2$, and finally~$RS_1$ to~$P_1S_1$ (see Figure~\ref{F:sectionPQRS}). The geodesics~$b_1$ and~$b_2$ intersect in four points, that we denote by~$A, B, C, D$ in such a way that both~$b_1$ and $b_2$ go through them in this order. The union~$b_1\cup b_2$ plays the role of the billiard trajectory in Section~\ref{S:pqr}. It divides~$\Orbpqrs$ into six regions: the $P$-part which has two edges and vertices~$A$ and $B$, the $Q$-part with two edges and vertices~$C$ and $D$, the $R$-part with vertices~$B$ and $C$, the $S$-part with vertices~$D$ and $A$, and two quadrangles, called~$T_1$ and $T_2$ in such a way that the two geodesics go clockwise around~$T_1$, and counter-clockwise around~$T_2$.

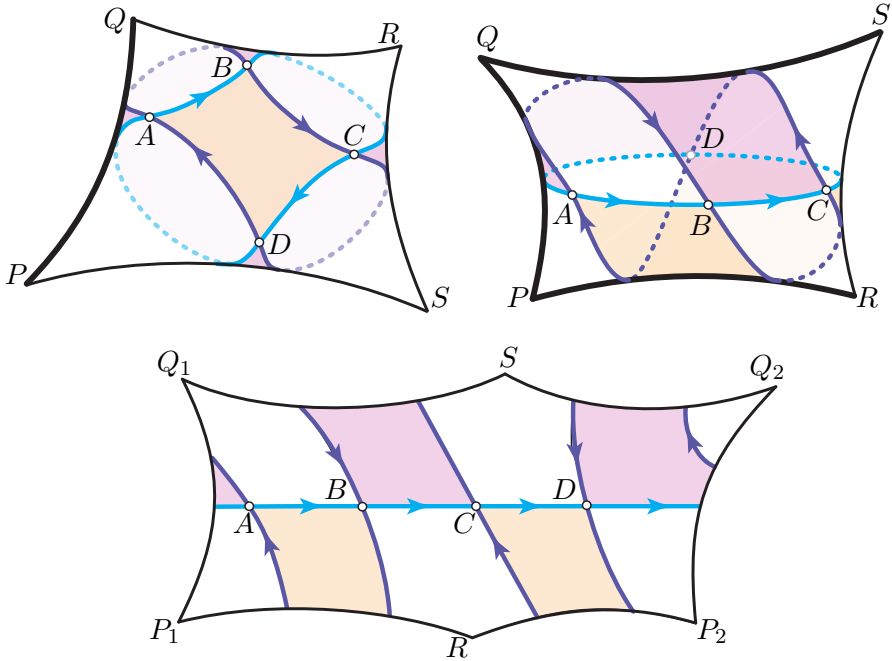
\begin{figure}[ht]
	\begin{picture}(120,87)(0,0)
	\put(0,0){\includegraphics*[width=.9\textwidth]{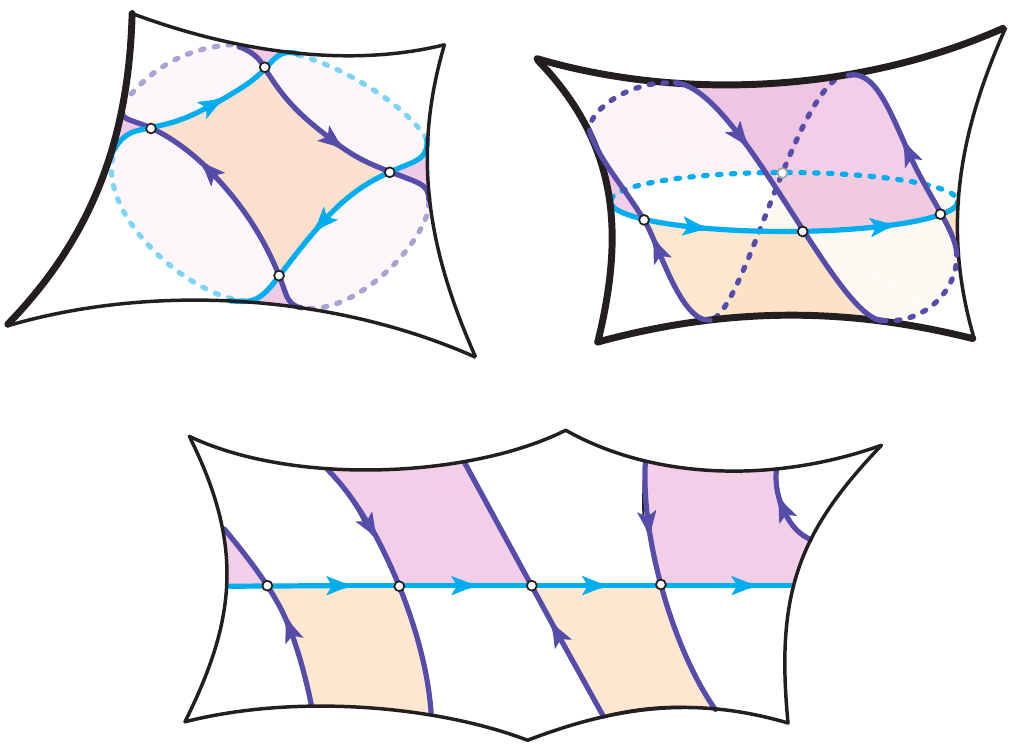}}
	\put(-2,48){$P$}
	\put(11,82){$Q$}
	\put(47,80){$R$}
	\put(54,45){$S$}
	\put(15,66.5){$A$}
	\put(25,75.5){$B$}
	\put(42.5,66.5){$C$}
	\put(32.5,52){$D$}
	\put(64,45){$P$}
	\put(60,79.5){$Q$}
	\put(110,45){$R$}
	\put(112,82.5){$S$}
	\put(70,56){$A$}
	\put(88,54.5){$B$}
	\put(103.5,56.5){$C$}
	\put(89.5,66){$D$}
	\put(17,1){$P_1$}
	\put(18,36.5){$Q_1$}
	\put(63,37){$S$}
	\put(56,-1.5){$R$}
	\put(89,1){$P_2$}
	\put(96,36){$Q_2$}
	\put(28,15){$A$}
	\put(40,19.5){$B$}
	\put(57,15){$C$}
	\put(70,19.5){$D$}
	\end{picture}
	\caption{\small Two views of an orbifold with four singular points~$P, Q, R, S$ and a fundamental domain obtained by cutting along $RP, PQ$ and $QS$. In blue the two geodesics~$b_1$ and $b_2$, which intersect at four points, called $A, B, C$ and $D$. The two quadrangles~$T_1, T_2$ are coloured.} 
	\label{F:sectionPQRS}
\end{figure}

Now we choose an oriented foliation $(\alpha_s(t))_{s,t\in(0,1)}$ of~$T_1\cup T_2$ by butterflies centered at~$A$, so that the orientation of the curves in both quadrangles coincides with the orientation of~$b_1$ and~$b_2$.

\begin{definition}\label{D:sectionPQRS}
We call~$\Spqrs^A$ the closure of the set of all unit tangent vectors to the family~$(\alpha_s(t))_{s,t\in(0,1)}$.
\end{definition}

The situation is similar to that of Sections~\ref{S:p2} and~\ref{S:pqr}. In particular taking the closure is equivalent to adding the tangent vectors to~$b_1$ and~$b_2$, the vectors at~$A$ that point into the~$Q$-part or into $T_1, T_2$, the vectors at~$B$ that point into the~$R$-part, the vectors at~$C$ that point into the~$S$-part, and the vectors at~$D$ that point into the~$P$-part (in red in Figure~\ref{F:sectionPQRS2}).

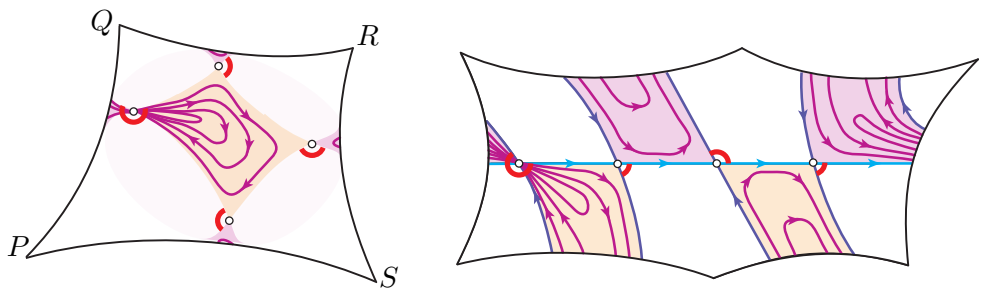
\begin{figure}[ht]
	\begin{picture}(120,40)(0,0)
	\put(0,0){\includegraphics*[width=\textwidth]{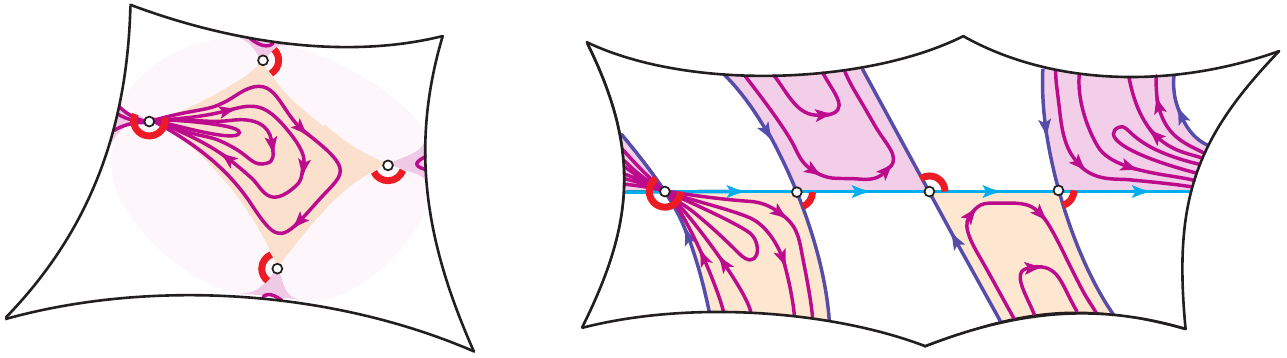}}
	\put(-2,4){$P$}
	\put(9,34){$Q$}
	\put(44,32){$R$}
	\put(47,0){$S$}
	\end{picture}
	\caption{\small The foliation of~$T_1\cup T_2$ by the family~$(\alpha_s(t))_{s,t\in(0,1)}$. The unit tangent vectors form the surface~$\Spqrs^A$.} 
	\label{F:sectionPQRS2}
\end{figure}

Now we define $c^{AB}$ to be the lift in~$\Spqrs^A$ of a loop in~$\Orbpqrs$ that winds counter-clockwise around the~$P$-part (for a similarity with the case of three singular points, one can for instance choose the loop so that~$c^{AB}$ is the set of all vectors of~$\Spqrs^A$ that point directly toward the point~$P$). Note that~$c^{AB}$ goes through the points~$A$ and $B$. Next, define~$c^{AD}$ to be the lift of a loop in~$\Orbpqrs$ that winds counter-clockwise around the~$Q$-part (again, one can choose the loop so that~$c^{AD}$ is the set of all vectors of~$\Spqrs^A$ whose opposite points toward~$Q$). Then define~$c^{AC}$ to be the lift of a closed loop that goes from~$A$ to~$RP_1$, from $RP_2$ to $C$ and then to $SQ_1$, and finally from $SQ_2$ to~$A$ in the fundamental domain.
We denote by $\wSpqrs^A$ the surface obtained by compactifying both boundaries of~$\Spqrs^A$ to a point.

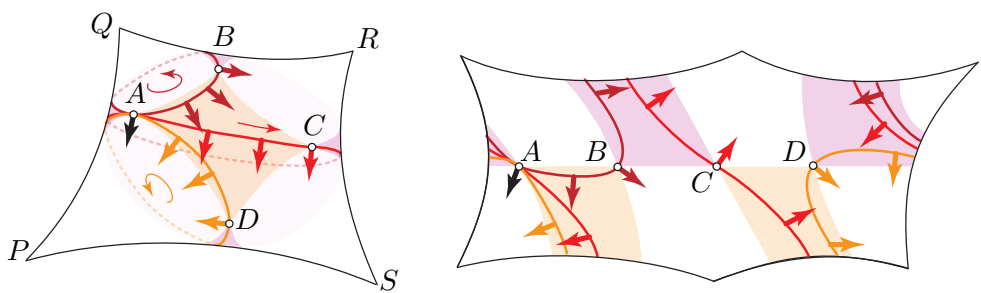
\begin{figure}[ht]
	\begin{picture}(120,40)(0,0)
	\put(0,0){\includegraphics*[width=\textwidth]{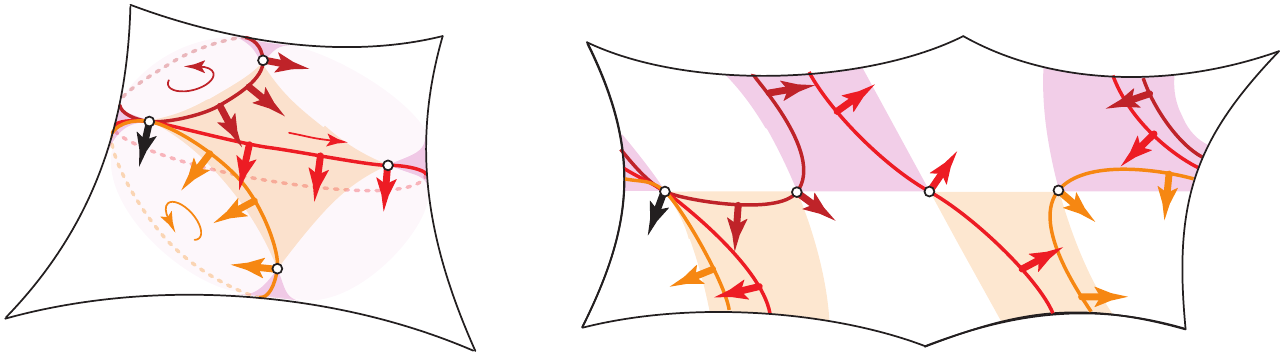}}
	\put(-2,4){$P$}
	\put(9,34){$Q$}
	\put(44,32){$R$}
	\put(47,0){$S$}
	\put(13.5,24.7){$A$}
	\put(25,32.3){$B$}
	\put(37,20.5){$C$}
	\put(28,8){$D$}
	\put(65.5,17){$A$}
	\put(74,17){$B$}
	\put(88,13){$C$}
	\put(100,17){$D$}
	\end{picture}
	\caption{\small The curves $c^{AB}, c^{AC}$ and~$c^{AD}$ on~$\Spqrs^A$. The black tangent vector is their common intersection point. The orientations are represented on the left by thinner arrows.} 
	\label{F:curvesPQRS}
\end{figure}

\begin{lemma}\label{L:SPQRS}
The surface~$\Spqrs^A$ has two boundary components, namely the lifts~$\tilde b_1$ and~$\tilde b_2$ of~$b_1$ and $b_2$. It is a genus one Birkhoff section for~$\flotpqrs$. The loops~$(c^{AD}, c^{AB})$ form a basis of~$H_1(\wSpqrs^A, \Z)$. We have the homological relation $c^{AC}=c^{AB} - c^{AD}$.
\end{lemma}

\begin{proof}
The idea is the same as for Lemma~\ref{L:SPQR}. 
Let~$\gamma$ be a geodesic in~$\Hy$ that is not in the~$\Gpqrs$-orbit of~$b_1$ or~$b_2$. As in Section~\ref{S:pqr}, we define the \emph{code}~$c(\gamma)$ of $\gamma$ to be the bi-infinite word, this time in the alphabet~$\{P, Q, R, S\}$, that lists the types of the regions crossed by~$\gamma$, forgetting about the quadrangles.  The curve $\gamma$ cannot cross more than two consecutive quadrangles, so that the code is indeed bi-infinite. For every factor~$QP, QR, QS, RP, RS, SP, QQ, RR, SS$ or $PP$, there exists exactly one point corresponding to that factor where~$\gamma$ is tangent to the family~$\alpha_s$. This can be seen on the left of Figure~\ref{F:sectionPQRS2} or in Figure~\ref{F:curvesPQRS}. These factors correspond to non-increasing factors wih respect to the order~$Q\ge R\ge S\ge P$. Therefore every bi-infinite word contains infinitely many such factors, so that the lift of~$\gamma$ in~$\U\Orbpqrs$ intersects $\Spqrs^A$ infinitely many times.

For the genus, we observe that $\Spqrs^A$ is made of the closures of two discs, corresponding to the lifts of the two quadrangles. With this decomposition, there are eight edges corresponding to the different segments of~$b_1$ and~$b_2$, plus four edges in the fibers of~$A, B, C$ and $D$. There are also eight vertices, two in each of the fibers of~$A, B, C$ and $D$. Adding all contributions, we obtain $-1$ for the Euler characteristics of~$\Spqr^A$, whence one for its genus. For the basis of~$H_1(\wSpqr^A, \Z)$, we observe as in Lemma~\ref{L:S2QR} that the considered loops intersect each other once.

The torus~$\Spqrs^A$ is displayed in Figure~\ref{F:TABCD}. There the two boundary components are small circles, and the intersection of the fibers of~$A, B, C, D$ with~$\Spqrs^A$ are segments connecting them. In~$\wSpqrs^A$, the curve~$c^{AB}$ is isotopic to the concatenation of the part of~$b_1$ between~$A$ and $B$, of the part of the fiber of~$B$ in~$\Spqrs^A$, of the part of~$b_2$ between~$B$ and~$A$, and, finally, of the part of the fiber of~$A$ in~$\Spqrs^A$. Therefore it is isotopic to the purple vertical loop as depicted on the figure. Similar arguments show that $c^{AD}$ is isotopic to the orange horizontal loop, and that $c^{AC}$ is isotopic to the red diagonal loop. Then the expected homological relation $c^{AC}=c^{AB} - c^{AD}$ follows.
\end{proof}

\begin{figure}[h]
	\begin{picture}(45,45)(0,0)
 	\put(-5,-3){\includegraphics*[width=.39\textwidth]{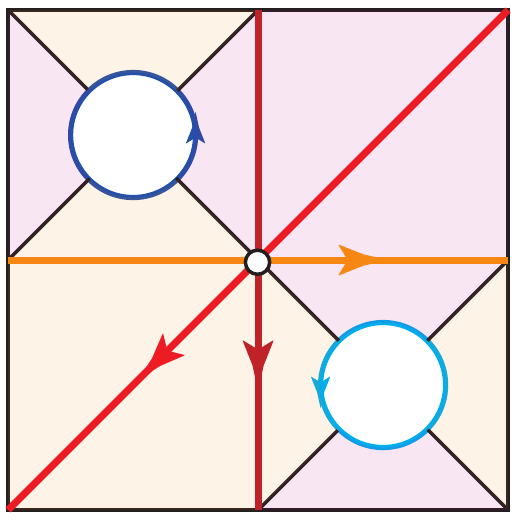}}
	\put(6,33){$b_2	$}
	\put(29.6,9.2){$b_1$}
	\put(14.5,27.5){$A$}
	\put(14.5,38){$B$}
	\put(24,0){$B$}
	\put(35,0){$C$}
	\put(-3,38){$C$}
	\put(39,15){$D$}
	\put(-3,26){$D$}
	\put(12.5,3){$c^{AB}$}
	\put(1,8){$c^{AC}$}
	\put(31,23){$c^{AD}$}
	\end{picture}
	\caption{\small }
	\label{F:TABCD}
\end{figure}

We define the surfaces~$\Spqrs^B$, $\Spqrs^C$ and~$\Spqrs^D$ similarly. These are also tori with two boundary components, and they are Birkhoff sections for~$\flotpqrs$. We also define the loops~$c^{BA}, c^{BC}$ and $c^{BD}$ in~$\Spqrs^B$, the loops~$c^{CA}, c^{CB}$ and $c^{CD}$ in~$\Spqrs^C$, and~$c^{DA}, c^{DB}$ and $c^{DC}$ in~$\Spqrs^D$. Then, for the same reason as above, the homological equalities~$c^{BD}=c^{BC} - c^{BA}$ in~$\wSpqrs^B$, $c^{CA}=c^{CD} - c^{CB}$ in~$\wSpqrs^C$ and $c^{DB}=c^{DA} - c^{DC}$ in~$\wSpqrs^D$ hold.

\subsection{First return maps}

For every tangent vector $v$ that lies in the surface~$\Spqrs^A$ and not in the fiber of~$C$ or~$D$, we define $\phi^A(v)$ to be the first intersection between the orbit of the geodesic flow starting from~$v$ and the surface~$\Spqrs^D$. For $v$ a tangent vector at~$C$ that points into the~$S$-part, or a tangent vector at~$B$ that points into the $R$-part, we define~$\phi^A(v)$ to be $v$ itself. In this way we have defined a map~$\phi^A$ from~$\Spqrs^A$ into~$\Spqrs^D$. 

\begin{lemma}\label{L:PQRSreturn}
The map~$\phi^A$ is a homeomorphism from~$\Spqrs^A$ to~$\Spqrs^D$. It is conjugated to the homeomorphism whose matrix in the bases~$(c^{AD}, c^{AB})$ and~$(c^{DC}, c^{DA})$ is~$\left(\begin{smallmatrix} 0 & -1 \\ 1 & p \end{smallmatrix}\right)$. 
\end{lemma}

\begin{proof}
The argument is similar to the one for Lemma~\ref{L:QRreturn} and~\ref{L:PQRreturn}. First, the continuity and the injectivity of~$\phi^A$ are straightforward.

Next, the same argument as in Lemma~\ref{L:PQRreturn} shows that the image of~$c^{AD}$ under~$\phi^A$ is isotopic to~$c^{DA}$. The point is then to determine the image of~$c^{AB}$. The argument is displayed in Figure~\ref{F:phiPQRS}, where we see that this image is isotopic to~$c^{DB} + (p{-}1)\, c^{DA}$. Finally, by Lemma~\ref{L:SPQRS} we have~$c^{DB} = c^{DA} - c^{DC}$, so that the latter sum is equal to~$-c^{DC} + p\,c^{DA}$. This gives the expected matrix. 
\end{proof}

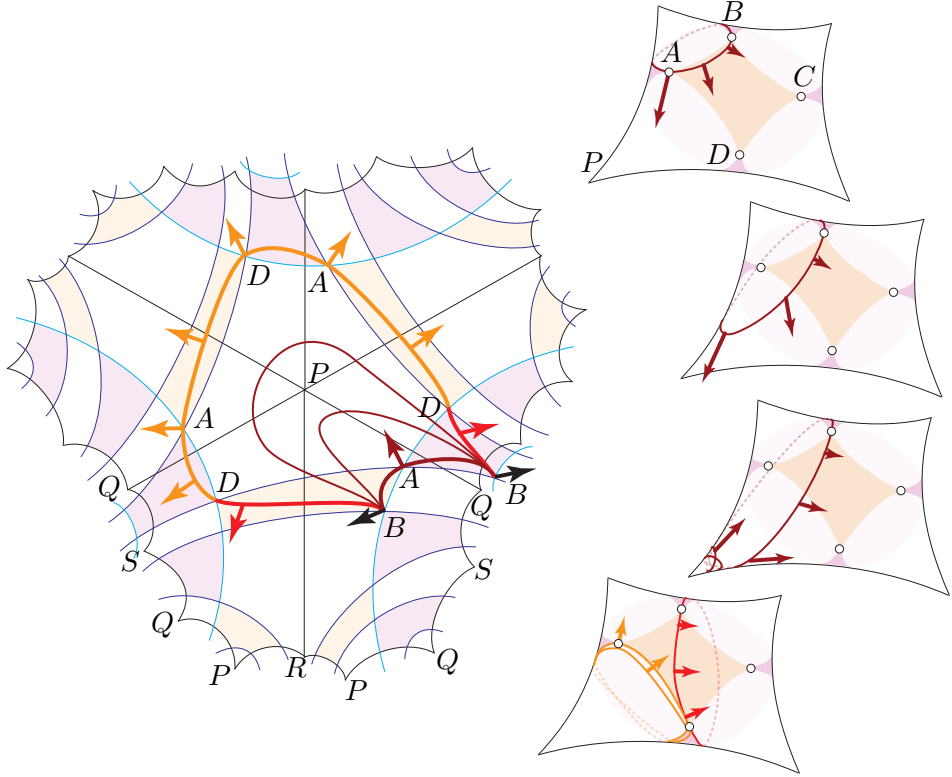
\begin{figure}[ht]
	\begin{picture}(130,110)(0,0)
	\put(0,0){\includegraphics*[width=\textwidth]{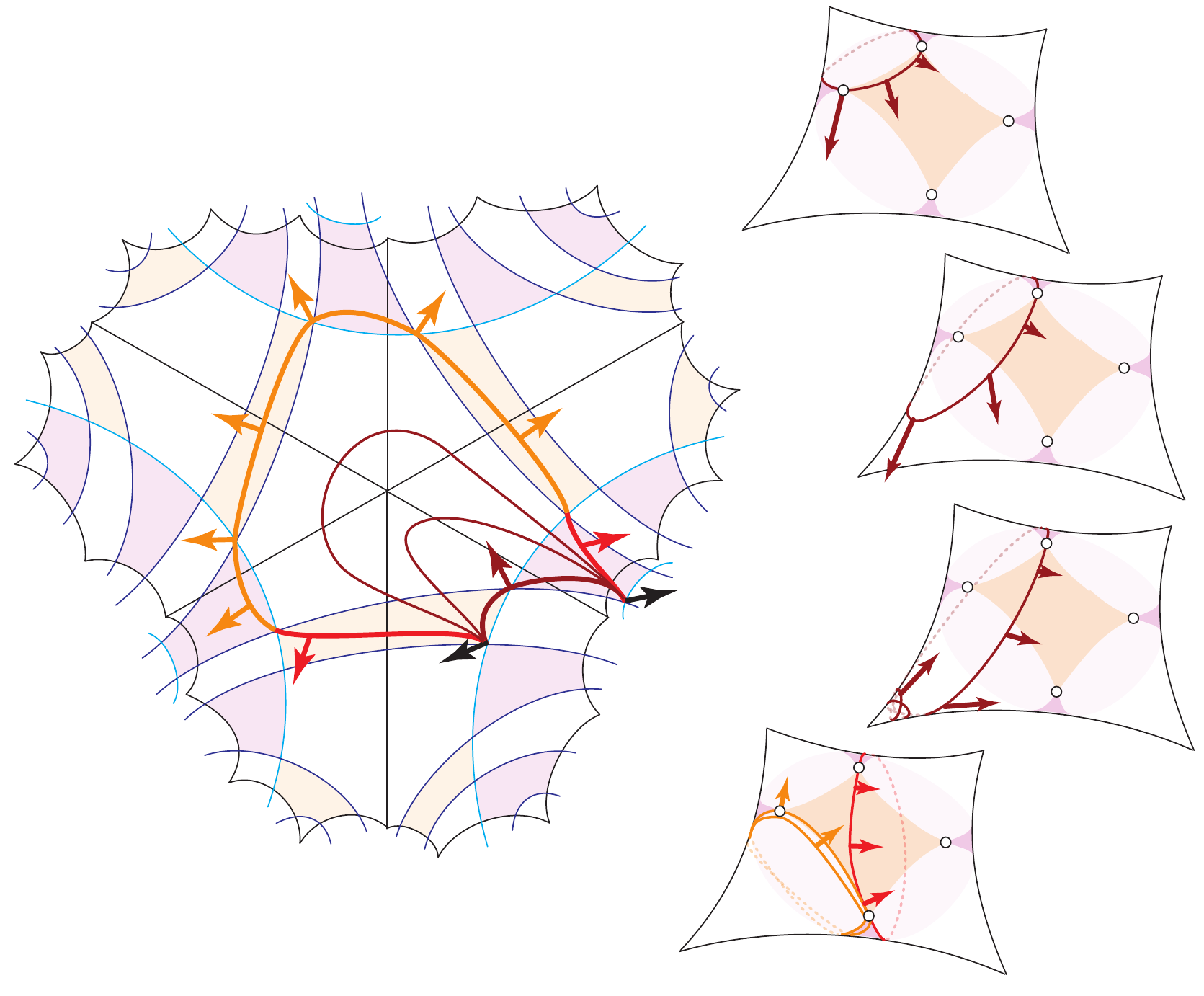}}
	\put(41,54){$P$}
	\put(13.5,39){$Q$}
	\put(16.5,29){$S$}
	\put(20.5,21){$Q$}
	\put(28,14){$P$}
	\put(38,15){$R$}
	\put(46,12){$P$}
	\put(58,16){$Q$}
	\put(63,28){$S$}
	\put(62.5,36.5){$Q$}
	\put(51,33.5){$B$}
	\put(53,40){$A$}
	\put(67,38){$B$}
	\put(29,39.5){$D$}
	\put(26,48){$A$}
	\put(33,67){$D$}
	\put(41,66){$A$}
	\put(55.5,49.5){$D$}
	\put(87.5,96.5){$A$}
	\put(95.5,102){$B$}
	\put(105,93.3){$C$}
	\put(93.5,83){$D$}
	\put(77,82){$P$}
	\end{picture}
	\caption{\small Transportation of the loop~$c^{AB}$ by the geodesic flow keeping the point that belongs to the fiber of~$B$ fixed. The image crosses the fiber of~$P$ (here we have~$p=3$). After some time, it reaches~$\Spqrs^D$ and is then isotopic to~$c^{DB} + (p{-}1)\, c^{DA}$.} 
	\label{F:phiPQRS}
\end{figure}

We define the maps~$\phi^B, \phi^C$ and~$\phi^D$ similarly. We can now conclude.

\begin{proof}[Proof of Proposition~\ref{T:FirstReturn} (last case)]
Consider the four Birkhoff sections~$\Spqrs^A$, $\Spqrs^B$, $\Spqrs^C$ and~$\Spqrs^D$ given by Lemma~\ref{L:SPQRS}. Starting from any point of~$\Spqrs^A$ and following the geodesic flow for some time (which is bounded, but not the same for all points), we reach the surface~$\Spqrs^D$, then~$\Spqrs^C$, then~$\Spqrs^B$, and then~$\Spqrs^A$ again. Therefore the first return map on~$\Spqrs^A$ is obtained by applying~$\phi^A$, then~$\phi^D$, $\phi^C$, and finally~$\phi^B$. In terms of matrices, and in the basis~$(c^{AD}, c^{AC})$, it is therefore conjugated to the product~$\left(\begin{smallmatrix} 0 & -1 \\ 1 & q \end{smallmatrix} \right) \left(\begin{smallmatrix} 0 & -1 \\ 1 & r \end{smallmatrix} \right) \left(\begin{smallmatrix} 0 & -1 \\ 1 & s \end{smallmatrix} \right) \left(\begin{smallmatrix} 0 & -1 \\ 1 & r \end{smallmatrix} \right)$. In terms of the standard generators~$X$ and~$Y$ of~$\SLZ$, the latter product is equal to $X^{-1}YX^{q-1}X^{-1}YX^{r-1}$ $X^{-1}YX^{s-1}X^{-1}YX^{p-1}$, which is conjugated to~$X^{p-2}YX^{q-2}YX^{r-2}YX^{s-2}Y$.

The other genus one Birkhoff sections are now obtained by changing the choice of the geodesics~$b_1$ and~$b_2$. For every cyclic ordering of the letters~$P, Q, R$ and $S$, we can find two geodesics whose union divides the orbifold~$\Orbpqrs$ into two quadrangles and four parts containing the singular points, so that the adjacencies between the different parts follow the cyclic order. Applying the same strategy as above then gives the expected monodromies.

In the case when some parameter equals $2$, say~$p$, the picture degenerates: the $P$-part of~$\Orbpqrs$ collapses to the point~$P$ and the quadrangles~$T_1$ and $T_2$ collapse to triangles. Nevertheless, one easily verifies that the construction and all subsequent observations remain valid, so that the result still holds.
\end{proof}

\section{Further questions}

Theorem~\ref{T:B} gives a positive answer to Ghys' Question~\ref{Q:Ghys} in several particular cases, but the general case remains open. The simplest cases for which the answer is unknown correspond to the conjugation classes described by words containing five~$X$ and five~$Y$.

\begin{question}
Is there a hyperbolic 2-orbifold whose geodesic flow admits a genus one Birkhoff section with first return map conjugated to~$\aut{X^5Y^5}$? or to $\aut{(XY)^5}$?
\end{question}

The construction of Birkhoff-Fried-Ghys-Hashiguchi shows that the answer to Question~\ref{Q:Ghys} is positive for the classes of the form~$(X^2Y^{g-1})^2$ with $g\ge 2$, and we observe that these classes all are eligible by Theorem~\ref{T:B}. By contrast, the construction by Brunella~\cite{Brunella} shows that the answer is also positive for the classes of the form~$X^2(X^2Y^{g-1})^2$, but the latter are not eligible by Theorem~\ref{T:B}, since they contain six~$X$. This naturally leads to
 
\begin{question}
Is there a simple construction of genus one sections for geodesic flows that simultaneously includes the examples of Birkhoff, the examples of Brunella, and the examples of Proposition~\ref{T:FirstReturn}?
\end{question}

In a different perspective, our construction shows that the geodesic flow on every compact hyperbolic orbifold with spherical base and three or four singular points admits Birkhoff sections of genus~one. Whether such a construction exists for every orbifold is not clear. In particular, the analog of our construction for spheres with~$k$ singular points gives sections of genus~$\lfloor (k-1)/2\rfloor$, which is larger than~$1$ for $k \ge 5$. Our attempts to modify the construction for a sphere with five singular points have failed so far, so that the following particular case of Fried's Question~\ref{Q:Fried} is still open:

\begin{question}
Does the geodesic flow on a sphere with five singular points admit a genus one Birkhoff section?
\end{question}

\bibliographystyle{alpha}

\end{document}